\documentclass[11pt, letterpaper]{article}
\usepackage{fullpage}
\usepackage{amsthm,amsmath,amssymb,amsfonts}

\usepackage[noabbrev,capitalize,nameinlink]{cleveref}
\crefname{equation}{}{}

\newtheorem{thm}{Theorem}[section]

\newtheorem{lem}[thm]{Lemma}
\newtheorem{cor}[thm]{Corollary}

\newcommand{\mc}[1]{\mathcal{#1}}
\newcommand{\mb}[1]{\mathbb{#1}}
\newcommand{\nib}[1]{\noindent {\bf #1}}

\newcommand{\bfl}[1]{\left\lfloor #1 \right\rfloor}
\newcommand{\bcl}[1]{\left\lceil #1 \right\rceil}

\newcommand{\sub}{\subset}

\newcommand{\sups}{\supseteq}

\newcommand{\sm}{\setminus}
\newcommand{\es}{\emptyset}

\newcommand{\aA}{\alpha}
\newcommand{\bB}{\beta}
\newcommand{\gG}{\gamma}
\newcommand{\dD}{\delta}
\newcommand{\eps}{\varepsilon}

\newcommand{\tT}{\theta}

\newcommand{\DD}{\Delta}

\newcommand{\TT}{\Theta}
\newcommand{\Ss}{\Sigma}

\DeclareMathOperator{\col}{col}
\DeclareMathOperator{\var}{var}

\title{The optimal edge-colouring threshold}

\author{Peter Keevash \thanks{Mathematical Institute, University of Oxford, UK. 
Supported by ERC Advanced Grant 883810.} }

\date{\vspace{-5ex}}
\begin{document}
\maketitle

\begin{abstract}
Consider any dense $r$-regular quasirandom $H \sub K_{n,n}$
and fix a set of $r$ colours. Let $L$ be a random list assignment
where each colour is available for each edge of $H$ with probability $p$.
We show that the threshold probability for $H$ 
to have a proper $L$-edge-colouring
is $p = \TT( \tfrac{\log n}{n} )$.
This answers a question of 
Kang, Kelly, K\"uhn, Methuku and Osthus.
We thus obtain the same threshold for
Steiner Triple Systems and Latin squares;
the latter answers a question of Johanssen from 2006.
\end{abstract}

\section{Introduction}

The theory of thresholds in random structures
has recently been revolutionised by the solution
of the Kahn-Kalai Conjecture \cite{KK},
at first in its fractional form
(a conjecture of Talagrand \cite{T}) 
by Frankston, Kahn, Narayanan and Park \cite{FKNP},
and then completely by Park and Pham \cite{PP}.
This recent theory easily implies many previously difficult 
results on thresholds, such as the celebrated solution 
by Johansson, Kahn and Vu \cite{JKV} of Shamir's problem
(reported by Erd\H{o}s \cite{E}) 
on the threshold for perfect matchings in hypergraphs.

Despite this progress, it still seems challenging to determine
thresholds for designs or design-like structures.
Such questions seem to have first been raised by Johansson \cite{J}, 
who asked for an analogue of Shamir's problem for Latin squares.
His formulation uses the identification of a Latin square of order $n$
with a triangle decomposition of $K_{n,n,n}$, and asks for the 
threshold probability $p$ for finding such a decomposition
when each triangle is available independently with probability $p$.
The natural implicit conjecture (by analogy with Shamir's problem)
is that the decomposition exists whp (with high probability)
when every edge is in a triangle, which occurs at $p = \TT(n^{-1}\log n)$.
This implicit conjecture was later explicitly made 
independently by Casselgren and H\"aggkvist \cite{CH}
(in the equivalent form of edge-colouring $K_{n,n}$ from random lists)
and by Luria and Simkin \cite{LS}.

The corresponding questions for designs were not posed explicitly 
until quite recently, no doubt because even the existence of general designs 
was unknown before \cite{K}, although the natural conjecture was clear 
to Kahn and Kalai (personal communication).
An explicit conjecture (in some cases) was formulated by Simkin \cite{S}.
While the general case remains wide open, there has been significant recent
progress on the threshold for Latin squares 
and also Steiner Triple Systems, i.e.\ the threshold probability 
for a random $3$-graph on $n$ vertices to contain
a Steiner Triple System (assuming $n \equiv 1,3$ mod $6$).

A recent breakthrough on these problems due to
Sah, Sawhney and Simkin \cite{SSS} gives an upper bound of $n^{-1+o(1)}$.
This was improved by Kang, Kelly, K\"uhn, Methuku and Osthus \cite{KKKMO} 
to $O(n^{-1} \log^2 n)$, which is within a factor $\log n$ of being optimal.
Both arguments use
the Frankston-Kahn-Narayanan-Park theorem and Iterative Absorption.
The approach taken by \cite{KKKMO} reduces both questions to the problem 
of constructing a sufficiently spread measure on optimal edge-colourings
of a regular nearly-complete bipartite graph.

\subsection{Results}

Our main theorem solves this spread measure problem 
(see the next subsection for definitions
of the terminology used in its statement).

\begin{thm} \label{thm:main}
For any $\dD>0$ there is $K>0$ such that 
if $H \sub K_{n,n}$ is $r$-regular and $(\dD,r/n)$-uniform with $r>\dD n$
then there is a $K/n$-spread probability distribution 
on optimal edge-colourings of $H$.
\end{thm}

\cref{thm:main} has the following consequences
which include resolutions of the conjectures discussed above
on the thresholds for Steiner Triple Systems and Latin squares.

\begin{cor} \label{cor:sts}
There is an absolute constant $C>0$ such that 
if $n \equiv 1,3$ mod $6$ and $p \ge \tfrac{C\log n}{n}$
then the Erd\H{o}s-R\'enyi $3$-graph $G^3(n,p)$
whp contains a Steiner Triple System of order $n$.
\end{cor}

\begin{cor} \label{cor:latin}
There is an absolute constant $C>0$ such that 
the following (easily equivalent) 
statements hold for $p \ge \tfrac{C\log n}{n}$.

1. An $n$ by $n$ square where $n$ symbols are   
each available for each cell independently with probability $p$ 
whp contains a Latin square of order $n$.

2. The Erd\H{o}s-R\'enyi $3$-partite $3$-graph $G^3(n,n,n;p)$
whp contains a triangle decomposition of
the complete $3$-partite graph $K_{n,n,n}$.

3. If $L$ is a random $(pn,n)$-list assignment for $E(K_{n,n})$ then 
whp there is an $L$-edge-colouring of $K_{n,n}$.
\end{cor}

\begin{cor} \label{cor:lists}
There is an absolute constant $C>0$ such that 
if $L$ is a random $(C\log n,2n-1)$-list assignment for $E(K_{2n})$ then 
whp there is an $L$-edge-colouring of $K_{2n}$.
\end{cor}

\begin{cor} \label{cor:threshold}
For any $\dD>0$ there is $C>0$ such that 
if $H \sub K_{n,n}$ is $r$-regular
and $(\dD,r/n)$-uniform with $r>\dD n$
and $L$ is a random $(C\log n,r)$-list assignment for $E(H)$
then whp there is an $L$-edge-colouring of $H$.
\end{cor}

The deductions of the these corollaries from \cref{thm:main}
are the same as the deductions in \cite{KKKMO} of the corresponding results
with an extra $\log n$ factor. 
Indeed, \cref{cor:threshold} (and so \cref{cor:latin})
is immediate from \cref{thm:main}
combined with the Frankston-Kahn-Narayanan-Park Theorem
(Talagrand's fractional Kahn-Kalai Conjecture).
Furthermore, the deductions of the results corresponding
to \cref{cor:sts} and \cref{cor:lists},
which are given by \cite[Theorems 1.6 and 1.7]{KKKMO},
are applicable for $p \ge \tfrac{C\log n}{n}$.

The main technical contribution of our paper lies in showing
that a carefully constructed random greedy process constructs 
a spread measure on approximately regular factorisations.
Specifically, given $H$ as in \cref{thm:main} we need to produce random
factorisations $(H_1,\dots,H_m)$ where each $H_c$ has all vertex degrees
$(1 \pm \dD)d$ for some $d=O(1)$ bounded independently of $n$.
Constructing such factorisations randomly is a delicate task
due to the lack of concentration of vertex degrees in random graphs
of density $O(1/n)$, so there are many subtleties in showing that
there is some `goldilocks zone' for constrained random greedy processes
that are sufficiently constrained to produce a factorisation
with the required regularity properties,
but not overly constrained so as impact spreadness.
 
\subsection{Definitions}

Let $H$ be a graph. A \emph{factorisation} $(H_1,\dots,H_m)$ of $H$
is a list of spanning subgraphs (each $V(H_c)=V(H)$)
such that $(E(H_1),\dots,E(H_m))$ is a partition of $E(H)$.
We will also think of a factorisation as a colouring of $E(H)$ 
where each edge in $H_c$ is assigned colour $c \in [m]$.
If every $H_c$ has some property $P$ we call it a \emph{$P$ factorisation}.
We will be particularly concerned with regular factorisations,
in which each $H_c$ is a spanning regular subgraph.
If all pieces are $1$-regular (matchings) 
we also speak of a \emph{$1$-factorisation} or \emph{optimal edge-colouring}.

We note that a $1$-factorisation can only exist if $H$ itself is regular.
Furthermore, if $H$ is bipartite and regular then it is well-known 
(and easy to see by Hall's Theorem) that $H$ has a $1$-factorisation.

A probability distribution on factorisations $(H_1,\dots,H_m)$ of $H$ 
is \emph{$q$-spread} if for any $S_1,\dots,S_m \sub E(H)$ 
we have $\mb{P}(\bigcap_c \{S_c \sub H_c\}) \le q^{\sum_c |S_c|}$.

If $H$ is bipartite then
any regular factorisation of $H$ can be refined into a $1$-factorisation,
so from any $q$-spread distribution on regular factorisations
we can obtain a $q$-spread distribution on $1$-factorisations.

For $D \sub \mb{N}$ we say that a graph $H$ is \emph{$D$-regular}
if all vertex degrees $|H(v)|$ are in $D$.
We will often take $D$ of the form $(1 \pm \dD)d := [d-\dD d,d+\dD d]$.

Let $H \sub K_{n,n}$ be a bipartite graph 
on $(U_1,U_2)$ with $|U_1|=|U_2|=n$.
We say that $H$ is $(\dD,p)$-uniform\footnote{
Here `regular' is more common in the literature,
but this word has several other uses in this paper,
so for clarity we use alternative terminology
(which also has some precedent in the literature). 
}
if for every $V_i \sub U_i$ with $|V_i| \ge \dD n$ for $i=1,2$ 
we have $|H[V_1,V_2]| = (1 \pm \dD)p|V_1||V_2|$.

We defer to the next section some further definitions 
of sparsity, density and quasirandomness
which are needed for our most general result Theorem \ref{thm:main+}.

We write $0 < a \ll b$ to indicate that the following statement
holds for any $b>0$ and $a \in (0,a_0(b))$ sufficiently small.
Hierarchies with more parameters are defined analogously.

We suppress notation for rounding to integers when this is obviously harmless.

\section{Preliminaries}

\subsection{Concentration of probability}

We will require the following well-known Chernoff bound, 
see e.g.~\cite[Theorems~2.1~and~2.10]{JLR}. 
\begin{lem}\label{lem:chernoff}
Let $\dD>0$ and $X$ be a binomial or hypergeometric random variable. Then
\[\Pr[X\le(1-\delta)\mb{E}X]\le\exp(-\delta^2\mb{E}X/2),
\qquad\Pr[X\ge (1+\delta)\mb{E}X]\le\exp(-\delta^2\mb{E}X/(2+\delta)).\]
\end{lem}

We often apply the Chernoff bound to a sum $X$ of random variables 
as in the following lemma that 
can be stochastically dominated by a binomial
via a straightforward coupling argument.

\begin{lem}\label{lem:bernoulli-domination}
Suppose the random variables $X_1,\dots,X_n$
and $Y_1,\dots,Y_n$ are $\{0,1\}$-valued,
with $Y_1,\dots,Y_n$ independent
and each $\mb{E}[X_i|X_1,\ldots,X_{i-1}]\le \mb{E}Y_i$.
Let $X = \sum_i X_i$ and $Y = \sum_i Y_i$.
Then for any $t\ge 0$ we have 
$\mb{P}(X \ge t) \le \mb{P}(Y \ge t)$.
\end{lem}

We also require Azuma's inequality for supermartingales,
see e.g~\cite[Theorem~2.25]{JLR}.

\begin{lem}\label{lem:azuma}
Let $(X_0,\dots,X_n)$ be a supermartingale sequence
satisfying $|X_k-X_{k-1}| \le c_k$ for some constants $c_k$.
Then $\mb{P}(X_n \ge X_0 + t) \le e^{-t^2/2V}$,
where $V = \sum_k c_k^2$ is the \emph{variance proxy}.
\end{lem}

\subsection{Spreadness}

A random subgraph $H$ of a graph $G$ is \emph{$q$-spread} 
if for any $S \sub E(G)$ we have 
$\mb{P}(S \sub H) \le q^{|S|}$.

We will repeatedly use the simple observation
that if $H$ is $q$-spread and
$E$ is an event for $H$ with $\mb{P}(E) \ge 1/2$
then $H \vert E$ is $2q$-spread.
Indeed, for any $S \sub E(G)$,
if $S=\es$ we have $\mb{P}(S \sub H \mid E)=1 = (2q)^{|S|}$,
or if $S \ne \es$ we have 
\[ \mb{P}(S \sub H \mid E) 
= \mb{P}(\{S \sub H\} \cap E)/\mb{P}(E)  
\le 2\mb{P}(S \sub H) \le 2q^{|S|} \le (2q)^{|S|}.\]

We also consider a more general spreadness notion
for random subgraphs of random subgraphs.
A random subgraph $H'$ of $H$ is \emph{conditionally $q'$-spread} 
if for any $S \sub E(G)$ we have 
$\mb{P}(S \sub H' \mid S \sub H) \le (q')^{|S|}$.

For example, if $H$ is $q$-spread then
we can also think of $H$ as being conditionally $q$-spread
in the `random' subgraph $G$ of $G$ in which 
every edge appears (with probability $1$).

\begin{lem}
Suppose $H' \sub H$ is conditionally $q'$-spread
and $H'' \sub H'$ is conditionally $q''$-spread.
Then $H'' \sub H$ is conditionally $q'q''$-spread.
\end{lem}

\begin{proof}
For any $S \sub E(G)$ we have
$\mb{P}(S \sub H'' \mid S \sub H)
= \mb{P}(S \sub H'' \mid S \sub H') \mb{P}(S \sub H' \mid S \sub H)
\le (q'')^{|S|} (q')^{|S|}$.
\end{proof}

Recall that
a probability distribution on factorisations $(H_1,\dots,H_m)$ of $H$ 
is \emph{$q$-spread} if for any $S_1,\dots,S_m \sub E(H)$ 
we have $\mb{P}(\bigcap_c \{S_c \sub H_c\}) \le q^{\sum_c |S_c|}$.

A probability distribution on factorisations $(H_1,\dots,H_m)$ 
of a random subgraph $H$ of $G$ is \emph{conditionally $q$-spread}
if for any $S_1,\dots,S_m \sub E(G)$ we have 
\[ \mb{P}(\bigcap_c \{S_c \sub H_c\} \mid \bigcap_c \{S_c \sub H\} ) \le q^{\sum_c |S_c|}. \]

\begin{lem} \label{lem:iterspread}
If $(H_1,\dots,H_m)$ is a conditionally $q$-spread factorisation of $H$
and $(H^i_1,\dots,H^i_{m_i})$ conditionally $q'$-spread factorisations of each $H_i$
that are conditionally independent given $(H_1,\dots,H_m)$
then their combination is a conditionally $qq'$-spread factorisation of $H$.
\end{lem}

\begin{proof}
We consider any $S^i_j \sub E(G)$ for $i \in [m]$, $j \in [m_i]$
and write $S_i = \bigcup_j S^i_j$. Then 
\begin{align*}
\mb{P}(\bigcap_{i,j} \{S^i_j \sub H^i_j\}) 
& = \mb{P}(\bigcap_i \{S_i \sub H_i\})
\mb{P}(\bigcap_{i,j} \{S^i_j \sub H^i_j\} \mid \bigcap_i \{S_i \sub H_i\}) \\
& \le q^{\sum_i |S_i|} \prod_i (q')^{\sum_j |S^i_j|}
= (qq')^{\sum_{i,j} |S^i_j|},
\end{align*}
where we used conditional independence then conditional spreadness.
\end{proof}

We conclude with a lemma  
of Pham, Sah, Sawhney and Simkin (see \cite[Lemma 4.1]{PSSS}) 
on spread measures for perfect matchings
in super-regular bipartite graphs.
Here we say that $H \sub K_{n,n}$ is $(\dD,p)$-super-regular
if $H$ is $(\dD,p)$-uniform and has minimum degree $\ge (1-\dD)pn$.

\begin{lem} \label{lem:spreadmatch} 
Let $\dD \ll 1/L \ll p$.
Suppose $G$ is a $(\dD,p)$-super-regular bipartite graph with parts of size $n$.
Then there is an $L/n$-spread distribution on perfect matchings of $G$.
\end{lem}

\subsection{Sparsity}

Suppose that $H$ is a graph on $n$ vertices.
We say that $H$ is \emph{$(\aA,\bB)$-sparse}
if for any $V \sub V(H)$ with $|V| \le \aA n$
we have $|H[V]| \le |V| \bB n$.

For any $V \sub V(H)$ 
we define a \emph{degeneracy order} $<$ on $V$
by starting with $H[V]$, repeatedly deleting 
any vertex of minimum degree in the remaining graph
until none remain, then defining $u<v$
whenever $v$ was deleted before $u$
(so the first deleted is last in the order).

Note that if $H$ is $(\aA,\bB)$-sparse then 
for any $V \sub V(H)$ with $|V| \le \aA n$
there is some $v \in V$ with $\le 2\bB n$ neighbours in $V$.
Thus for any degeneracy order $<$ on $V$,
any $v \in V$ has $\le 2\bB n$ neighbours $u \in V$ with $u<v$.

\begin{lem} \label{lem:sparse}
Let $C>1$, $\aA \in (0,1/4)$, $\gG > C(20\aA)^{1/5}$ and $r>5/\gG$.
Suppose $H$ is sampled from a $Cr/n$-spread distribution
on subgraphs of $K_{n,n}$. Then $H$ is $(\aA,\gG r/n)$-sparse 
with failure probability $< n^{-14}$, say.
\end{lem}

Note that for example that if $r>n/C$ then 
the distribution could be supported on a single graph $H$.
Here $\gG r/n > \gG/C > (20\aA)^{1/5}$
and for any $V \sub V(H)$ with $|V| < 2\aA n$
we have $|H[V]| \le |V|^2/4 < |V| \aA n < |V| \gG r$,
so $H$ is trivially $(\aA,\gG r/n)$-sparse.

\begin{proof}
We apply a union bound to estimate the probability
that there is some $V \sub V(H)$ with $|V| = k < 2\aA n$
such that $|H[V]| \ge \gG r k$.
Note that this is only possible if $k > \gG r$.
For fixed $k$, we bound the failure probability by
$\tbinom{2n}{k} \tbinom{k^2/4}{\gG r k} (Cr/n)^{\gG r k}
 < (2en/k)^k (k/r\gG \cdot Cr/n)^{\gG r k}
 < ( 2en/k \cdot (kC/n\gG)^5 )^k < (k/n)^{3k}$,
using $k/n < 2\aA < .1(\gG/C)^5$ and $\gG r > 5$.
Summing over $k$ with $2\aA n > k > \gG r > 5$ 
gives failure probability $< n^{-14}$.
\end{proof}

We also require the following similar lemma controlling small sets
in which many pairs have distinct common neighbours.

\begin{lem} \label{lem:sparse2}
Let $C,r>1$ and $n > (Cr)^8$.
Suppose $H$ is sampled from a $Cr/n$-spread distribution
on subgraphs of $K_{n,n}$. 
Then with failure probability $< n^{-40}$
we do not have disjoint $S,T \sub V(H)$ with $|T|=20|S| < \sqrt{n}$
and distinct pairs $(s^1_t s^2_t: t \in T)$ such that
each $s^i_t \in S$ and $ts^i_t \in E(H)$.
\end{lem}

\begin{proof}
We apply a union bound to estimate the probability 
of having such $S,T$ with $|S|=k$.
Note that we can only have $20k$ distinct pairs in $S$ if $k>40$.
For fixed $k$, we bound the probability by
$\tbinom{2n}{k} \tbinom{2n}{20 k} (k^2/4)^{20 k} (Cr/n)^{40 k}
 < (n/k)^k (kn \cdot (Cr/n)^2)^{20 k}
 = ( n/k \cdot (C^2 r^2 k/n)^{20} )^k < n^{-4k}$,
using $k < \sqrt{n}$ and $C^2 r^2 < n^{1/4}$. The lemma follows.
\end{proof}

\subsection{Quasirandomness}

Let $H$ be a bipartite graph on $(U_1,U_2)$ with $|U_1|=|U_2|=n$.
We say that $H$ is \emph{$(\dD,\dD',p)$-dense}
if for every $V_i \sub U_i$ with $|V_i| \ge \dD' n$ for $i=1,2$ 
we have $|H[V_1,V_2]| \ge (1-\dD)p|V_1||V_2|$.
We say that $H$ is \emph{$(\dD,\dD',\eta,p)$-quasirandom}
if it is $(1 \pm \dD)pn$-regular,
$(3\dD',\eta p)$-sparse and $(\dD,\dD',p)$-dense.

The above is a slight generalisation of a definition from \cite{KKKMO}:
$(\dD,p)$-quasirandom in their terminology follows from
$(\dD,\dD',1/9,p)$-quasirandom for some $\dD' \le \dD$.

Note that if $\dD' \ll \eta \ll p, \dD$
and $H \sub K_{n,n}$ is regular and $(\dD,p)$-uniform
then $H$ is $(\dD,\dD',\eta,p)$-quasirandom. 
Indeed, approximate regularity follows from regularity,
density follows from uniformity, and for sparsity we note that
if $V \sub V(H)$ with $|V| \le 6\dD' n$
then $|H[V]| \le |V|^2 \le |V| \eta p n$.

We require the following lemma from \cite{KKKMO}
(the maximum degree bound $\DD(R \cap H)<4\dD pn$
is not explicitly stated there, 
but follows from the proof).

\begin{lem} \label{lem:reg} \cite[Lemma 4.4]{KKKMO}
Let $1/n \ll \dD \ll 1$. 
Suppose $H \sub K_{n,n}$ is $(\dD,p)$-quasirandom
and $L \sub K_{n,n} \sm H$ is $(x \pm \dD)pn$-regular,
for some $p$ and $x$.
Then there exists $R \sub H \cup L$ such that 
$R \sups L$ is spanning and regular
and $\DD(R \cap H)<4\dD pn$. 
\end{lem}

\section{Iterative absorption}

Our main result \cref{thm:main} follows from the following
stronger version in which there is \emph{no non-trivial assumption} on $r$.
We also weaken the uniformity assumption to a density assumption.

\begin{thm} \label{thm:main+}
Let $1/K \ll \dD' \ll 1/C \ll \dD \ll 1$ and $r \in [n]$.
If $H$ is a sample from a $Cr/n$-spread distribution on
$r$-regular $(\dD,\dD',r/n)$-dense subgraphs of $K_{n,n}$
then there is a `good' event $G$ for $H$ with $\mb{P}(G)>1-n^{-8}$
and a $K/n$-spread probability distribution 
on $1$-factorisations of $H \vert G$.
\end{thm}

As any distribution is trivially $1$-spread,
by applying \cref{thm:main+} 
to the distribution that always outputs $H$,
for which we must have $\mb{P}(G)=1$,
we have the following corollary
which implies \cref{thm:main}.

\begin{cor}
Let $1/K \ll \dD' \ll \aA \ll \dD \ll 1$.
If $H \sub K_{n,n}$ is $r$-regular 
and $(\dD,\dD',r/n)$-dense with $r \ge \aA n$
then there is a $K/n$-spread probability distribution 
on $1$-factorisations of $H$.
\end{cor}

In this section we will deduce \cref{thm:main+} 
from the following lemma, 
whose proof is deferred to the next section.

\begin{lem} \label{lem:rough}
Let $1/K \ll 1/d \ll \dD' \ll \eta \ll 1/C \ll \dD \ll 1$.
Suppose $m,n \in \mb{N}$ with $n \ge K$. 
Let $H$ be sampled from a $Cdm/n$-spread distribution $\mc{D}$
on $(1 \pm \dD)dm$-regular subgraphs of $K_{n,n}$.

1. There is a good event $G$ for $H$ with $\mb{P}(G) > 1-n^{-9}$
and a conditionally $C/m$-spread probability distribution $\mc{D}'$
on $(1 \pm 2\dD)d$-regular factorisations $(F_1,\dots,F_m)$ of $H \vert G$.

2. If $\mc{D}$ is supported on
$(\dD,\dD',dm/n)$-dense subgraphs of $K_{n,n}$
then we can take $\mc{D}'$ supported on
$(2\dD,\dD',\eta,d/n)$-quasirandom factorisations $(F_1,\dots,F_m)$ of $H$.
\end{lem}

The remainder of this section constitutes
the proof of \cref{thm:main+} assuming \cref{lem:rough}.

Let $1/K \ll 1/d  \ll \dD' \ll \eta \ll 1/C  \ll \dD  \ll 1$ and $r \in [n]$.

Suppose $H$ is a sample from a $Cr/n$-spread distribution on
$r$-regular $(\dD,\dD',r/n)$-dense subgraphs of $K_{n,n}$.
We denote the parts of $K_{n,n}$ by $(U_1,U_2)$.

We need to give a $K/n$-spread probability distribution 
on $1$-factorisations of $H \vert G$, for some good event $G$.
We can assume $n \ge K$ and $r \ge K/C$, 
by considering the $\min\{Cr/n,1\}$-spread distribution
outputting an arbitrary $1$-factorisation of $H$
(which exists as $H$ is a regular bipartite graph).

We choose $d$ above (not necessarily an integer) 
so that $m := r/d \in \mb{N}$ is of the form 
$m = 2^{\ell}-1$ with $\ell \in \mb{N}$.

Our good event $G$ will be contained in
the event $G_0$ that $H$ is $(3\dD',\eta r/n)$-sparse,
and so $H \vert G_0$ is $(\dD,\dD',\eta,r/n)$-quasirandom. 
We have $\mb{P}(G_0) > 1 - n^{-14}$
by \cref{lem:sparse} applied with
$(C,3\dD',r,\eta)$ in place of $(C,\aA,r,\gG)$,
which is valid as $\eta r \ge \eta K/C > 5$
and $\eta > C(60\dD')^{1/5}$,
using $1/K \ll \dD' \ll \eta \ll 1/C$.

We apply \cref{lem:rough}.2 to $H \vert G_0$
to obtain an event $G_1$ with $\mb{P}(G_1) > 1 - n^{-9}$
and a conditionally $2C/m$-spread probability distribution 
on $(2\dD,\dD',\eta,d/n)$-quasirandom factorisations 
$F = (H_{i,j}: i \in [\ell], j \in [2^{\ell-i}])$ 
of $H \vert G$, where $G = G_0 \cap G_1$;
here we have renamed the output of the lemma
similarly to the `edge-vortex' in \cite{KKKMO}
for convenient use in the iterative absorption algorithm.
Note that
\begin{equation} \label{eq:F}
\text{As $H$ is $Cr/n$-spread, $F$ is $2C^2 d/n$-spread.} 
\end{equation}

We write $H_i = \bigcup \{ H_{i,j}: j \in [2^{\ell-i}] \}$ 
for each $i \in [\ell]$.

\begin{lem} \label{lem:Hi}
Each $H_i$ is a sample from a $2C^2 2^{\ell-i}d/n$-spread distribution
on $(2\dD,\dD',\eta,2^{\ell-i}d/n)$-quasirandom subgraphs of $K_{n,n}$.
Moreover, for any $S_i \sub E(K_{n,n})$ for $i \in [\ell]$ we have
$\mb{P}(\bigcap_i \{S_i \sub H_i\}) < \prod_i (2C^2 2^{\ell-i}d/n)^{|S_i|}$.
\end{lem}

\begin{proof}
Firstly, each $H_{i,j}$ is $(1 \pm 2\dD)d$-regular,
so each $H_i$ is $(1 \pm 2\dD)2^{\ell-i}d$-regular.
Next consider any $V_i \sub U_i$ for $i=1,2$.
If $|V_1|,|V_2| \ge \dD' n$ then
each $|H_{i,j}[V_1,V_2]| \ge (1-\dD)|V_1||V_2|d/n$,
so each $|H_i[V_1,V_2]| \ge (1-\dD)|V_1||V_2|2^{\ell-i}d/n$.
If $\dD' n>|V_1|>|V_2|/2$ then
each $|H_{i,j}[V_1,V_2]| \le \eta d|V_1|$,
so each $|H_i[V_1,V_2]| \le \eta 2^{\ell-i}d|V_1|$,
and similarly swapping subscripts $1$ and $2$.

It remains to prove the `moreover' statement
(which also implies the spreadness of each $H_i$). 
We partition each event $\{S_i \sub H_i\}$ into $(2^{\ell-i})^{|S_i|}$
events $\mc{A}_{P^i} = \bigcap_j \{ S_{i,j} \sub H_{i,j} \}$
for all partitions $P^i = (S_{i,j}: j \in [2^{\ell-i}])$ of $S_i$.
For $P=(P^1,\dots,P^\ell)$ let $\mc{A}_P = \bigcap_i \mc{A}_{P^i}$.
By spreadness of $F=(H_{i,j})$, see \eqref{eq:F}, we have
$\mb{P}(\mc{A}_P) \le (2C^2 d/n)^{\sum_{i,j} |S_{i,j}|}$.
Summing over all $P$ gives the required bound
on $\mb{P}(\bigcap_i \{S_i \sub H_i\})$.
\end{proof}

Now we are ready to describe the Iterative Absorption Algorithm,
which constructs a factorisation of $H$ into regular subgraphs
$\{ R_{i,j}: i \in [\ell], j \in [2^{\ell-i}] \}$.

\medskip

\nib{Iterative Absorption Algorithm}

\medskip

Step $i.0$: At the start of each step $i \in [\ell-1]$
we are given $R_i = \bigcup_j R_{i,j}$, where $R_1:=\es$, 
with leftover $L_i := H_i \sm R_i$. 

Step $i.1$: We abort unless we have the good event $G'_i$ that
$L_i$ is $(1 \pm 2.1\dD) 2^{\ell-i-1} d_i$-regular 
for some $d_i \in (1 \pm 40\dD) 2d$
and satisfies the good event $G_i$ for \cref{lem:rough}.1,
applied with $(d_i,2^{\ell-i-1},2C^2,2\dD)$ in place of $(d,m,C,\dD)$,
where we bound the spreadness of $L_i$ by that of $H_i$, using \cref{lem:Hi}.
Apply \cref{lem:rough}.1 to obtain
a conditionally $2C^2 / 2^{\ell-i-1}$-spread probability distribution 
on $(1 \pm 4.2\dD)d_i$-regular factorisations
$(L_{i,j}: j \in [2^{\ell-i-1}])$ of $L_i \vert G'_i$.

Step $i.2$: We apply \cref{lem:reg} 
with $(H_{i+1,j},L_{i,j},9\dD,d/n,d_i/d)$
in place of $(H,L,\dD,p,x)$
to find spanning regular subgraphs $R_{i+1,j}$ for $j \in [2^{\ell-i-1}]$
with $L_{i,j} \sub R_{i+1,j} \sub L_{i,j} \cup H_{i+1,j}$ 
and $\DD(R_{i+1,j} \cap H_{i+1,j}) \le 38\dD d$. 
If $i<\ell-2$ we now go to step $i+1$.
If $i=\ell-2$ we let 
$R_{\ell,1} = H \sm \bigcup_{i \in [\ell-1]} R_i$ 
and stop.

\begin{lem}
The algorithm aborts with probability $<n^{-8}$.
\end{lem}

\begin{proof}
Assuming \cref{lem:rough},
the event $\bigcap_i G_i$ has failure probability $<n^{-8}$.
It only remains to justify the regularity of $L_i$ in step $i.1$,
which we will show holds deterministically given that 
\cref{lem:rough} could be applied at previous steps.

For $i=1$ this holds by \cref{lem:Hi} as $L_1 = H_1$ 
is $(2\dD,\dD',\eta,2^{\ell-1}d/n)$-quasirandom.

Now suppose inductively we have $L_{i-1}$ for some $i \ge 2$ 
that is $(1 \pm 2.1\dD) 2^{\ell-i} d_{i-1}$-regular 
for some $d_{i-1} \in (1 \pm 40\dD) 2d$. 

In step $(i-1).1$, 
applying \cref{lem:rough} gives a distribution on
$(1 \pm 4.2\dD)d_{i-1}$-regular factorisations
$(L_{i-1,j}: j \in [2^{\ell-i}])$ of $L_{i-1} \vert G'_{i-1}$.

In step $(i-1).2$, applying \cref{lem:reg} gives 
spanning regular subgraphs $R_{i,j}$ for $j \in [2^{\ell-i}]$
with $L_{i-1,j} \sub R_{i,j} \sub L_{i-1,j} \cup H_{i,j}$ 
and $\DD(R_{i,j} \cap H_{i,j}) < 38\dD d$. 

Now note that $L_i$ is obtained from $H$ by deleting 
all $R_{i',j}$ with $i' \le i$ and all $H_{i',j}$ with $i'>i$,
where the $R_{i',j}$ are regular 
and each $H_{i',j}$ is $(1 \pm 2\dD)d$-regular.
Thus all degrees in $L_i$ differ by at most 
$\sum_{i'>i} 2^{\ell-i'} 2\dD d < 2^{\ell-i+1} \dD d$.

Furthermore $L_i = H_i \sm R_i$ has minimum degree $D_i$
with $D_i/2^{\ell-i} \in [(1-2\dD)d - 38 \dD d, (1+2\dD)d]$.
Let $d_i := D_i/2^{\ell-i-1} \in (1 \pm 40\dD) 2d$.
Then $L_i$ is $(1 \pm 2.1\dD) 2^{\ell-i-1} d_i$-regular, as $\dD \ll 1$. 
This completes the induction, so the lemma follows.
\end{proof}

The following lemma will complete the proof 
of \cref{thm:main+} assuming \cref{lem:rough}.

\begin{lem}
The algorithm conditioned on not aborting
outputs a $K/n$-spread probability distribution 
on regular factorisations of $H$.
\end{lem}

\begin{proof}
By construction, if the algorithm does not abort then it
constructs a factorisation of $H$ into regular subgraphs
$\{ R_{i,j}: i \in [\ell], j \in [2^{\ell-i}] \}$.
For spreadness, consider any sets of edges
$S = ( S_{i,j}: i \in [\ell], j \in [2^{\ell-i}] )$.
Let $\mc{A}_S$ be the event that the algorithm does not abort
and all $S_{i,j} \sub R_{i,j}$.
It suffices to show $\mb{P}(\mc{A}_S) \le (K/2n)^{\sum_{i,j} |S_{i,j}|}$.
As $R_1=\es$ we can assume all $S_{1,j}=\es$.

We partition $\mc{A}_S$ into events $\mc{A}_{S,T}$
for all $T = ( T_{i,j}: i \in [\ell], j \in [2^{\ell-i}] )$
with each $T_{i,j} \sub S_{i,j}$,
where $\mc{A}_{S,T}$ is the event that all
$T_{i,j} \sub R_{i,j} \cap H_{i,j}$
and $S_{i,j} \sm T_{i,j} \sub L_{i-1,j}$
(setting $L_{0,j}=\es$ for the trivial case $i=1$).

We write $\mc{A}_{S,T} = \mc{A}'_{S,T} \cap \bigcap_i \mc{A}^i_{S,T}$,
where $\mc{A}'_{S,T}$ is the event that all
$T_{i,j} \sub R_{i,j} \cap H_{i,j}$
and $S_{i,j} \sm T_{i,j} \sub H_{i-1}$,
and each $\mc{A}^i_{S,T}$ is the event that 
$S_{i,j} \sm T_{i,j} \sub L_{i-1,j}$ for all $j$.

By spreadness of $F=(H_{i,j})$, see \eqref{eq:F}, we have
$\mb{P}(\mc{A}'_{S,T}) < (2C^2 d/n)^{\sum_{i,j} |T_{i,j}|} 
 \prod_i (2^{\ell-i+1} 2C^2 d/n)^{|S_{i,j} \sm T_{i,j}|}$.
 
For each $i \in [\ell-1]$, writing
$\mc{E}^i_{S,T} = \mc{A}'_{S,T} \cap \bigcap_{i'<i} \mc{A}^{i'}_{S,T}$,
by conditional spreadness of $(L_{i,j})$ in Step i.1, we have
$\mb{P}(\mc{A}^i_{S,T} \mid \mc{E}^i_{S,T}) 
\le \mb{P}(\mc{A}^i_{S,T} \mid \bigcap_j \{ S_{i,j} \sm T_{i,j} \sub L_{i-1} \} ) 
\le (2C^2 / 2^{\ell-i})^{\sum_j |S_{i,j} \sm T_{i,j}|} $.
Thus
\begin{align*} 
& \mb{P}(\mc{A}_{S,T}) 
= \mb{P}(\mc{A}'_{S,T}) \prod_i \mb{P}(\mc{A}^i_{S,T} \mid \mc{E}^i_{S,T}) \\
& \le (2C^2 d/n)^{\sum_{i,j} |T_{i,j}|} 
 \prod_i (2^{\ell-i+1} 2C^2 d/n)^{|S_{i,j} \sm T_{i,j}|}
 \prod_i (2C^2 / 2^{\ell-i})^{\sum_j |S_{i,j} \sm T_{i,j}|} \\
& \le (100C^4 d/n)^{\sum_{i,j} |S_{i,j}|}.
\end{align*}

Summing over $T$, the lemma follows for $K > C^5 d$, say. 
\end{proof}

\section{Quasirandom factorisations}

In this section we prove \cref{lem:rough},
which completes the proof of our main theorem.
Let 
\[1/K \ll 1/d \ll \tT \ll \dD' \ll \eta \ll 1/C \ll \eps \ll \dD \ll 1/L \ll 1.\]

Suppose $m,n \in \mb{N}$ with $n \ge K$.
Let $H$ be sampled from a $Cdm/n$-spread distribution $\mc{D}$
on $(1 \pm \dD)dm$-regular subgraphs of $K_{n,n}$.
We will prove \cref{lem:rough} via a random greedy algorithm
described in the next subsection, with the following two properties
(we now rename $F_c$ as $H_c$).

\medskip

1. There is a good event $G$ for $H$ with $\mb{P}(G) > 1-n^{-9}$,
such that conditional on the algorithm not aborting, the output is
a conditionally $C/m$-spread probability distribution $\mc{D}'$
on $(1 \pm 2\dD)d$-regular factorisations $(H_1,\dots,H_m)$ of $H \vert G$.

\medskip

2. If $\mc{D}$ is supported on
$(\dD,\dD',dm/n)$-dense subgraphs of $K_{n,n}$
then $\mc{D}'$ is supported on
$(2\dD,\dD',\eta,d/n)$-quasirandom factorisations $(H_1,\dots,H_m)$ of $H$.

\medskip

First we give some reductions that allow
us to assume that the number of colours $m$ is quite large 
and the degree $dm$ is small compared with $n$.

\begin{lem} \label{lem:reduce}
\cref{lem:rough} follows from itself
assuming $dm < 2\log^4 n$ and $m \ge \sqrt{d}$ and replacing 
$(C,2\dD,n^{-9})$ by $(.1C,1.9\dD,n^{-10})$ in its conclusion.
\end{lem}

\begin{proof}
Suppose $dm \ge 2\log^4 n$, consider any $m' \in \mb{N}$
with $dm/m' \ge \log^4 n$ and a uniformly random factorisation
$(H'_1,\dots,H'_{m'})$ of $H$. This is conditionally $1/m'$-spread,
so each $H'_c$ is $Cdm/m'n$-spread. Also, by Chernoff
each $H'_c$ is $(1 \pm \dD \pm \log^{-2} n)dm/m'$-regular
with failure probability $\exp -\TT(\log^2 n)$. 
Furthermore, if $H$ is $(\dD,\dD',dm/n)$-dense then 
for every $V_i \sub U_i$ with $|V_i| \ge \dD' n$ for $i=1,2$ 
we have $|H[V_1,V_2]| \ge (1-\dD)|V_1||V_2|dm/n$,
so each $|H'_c[V_1,V_2]|$ is binomial with mean
$\ge (1-\dD)|V_1||V_2|dm/m'n$; then by Chernoff each
$|H'_c[V_1,V_2]| \ge (1-\dD-\log^{-2} n)|V_1||V_2|dm/m'n$ 
with failure probability $<e^{-n\log n}$, say,
by a union bound over $<n$ colours $c$
and $<4^{2n}$ choices of $V_1,V_2$.

We can suppose $d < \log^4 n$, 
otherwise choosing $m'=m$ above completes the proof.
Then we can fix $m' \le m/2$ with $dm/m' \in [\log^4 n,2\log^4 n]$.
We fix $m_c \in \{ \bfl{m/m'}, \bcl{m/m'} \}$ for $c \in [m']$
with $\sum_c m_c = m$ and consider a random factorisation
$(H'_1,\dots,H'_{m'})$ where independently uniformly at random
each edge samples $t_e \in [m]$ and is included in $H'_c$
for the smallest $c$ with $\sum_{c' \le c} m_{c'} \ge t_e$.
Thus each $H'_c$ is a random subgraph of $H$ where each
edge appears independently with probability $m_c/m$. 
Similarly to above, we have events $G'_c$ 
with failure probability $\exp -\TT(\log^2 n)$
such that each $H'_c \vert G'_c$ is a sample from
a $Cdm_c/n$-spread distribution
on $(1 \pm \dD \pm \log^{-2} n)dm_c$-regular graphs,
where if $H$ is $(\dD,\dD',dm/n)$-dense 
then $H'_c$ is $(\dD + \log^{-2} n,\dD',dm_c/n)$-dense. 

Now according to the assumptions of the lemma
we can apply \cref{lem:rough} to each $H'_c \vert G'_c$,
obtaining good events $G_c$ with $\mb{P}(G_c) > 1-n^{-10}$,
and conditionally $.1C/m_c$-spread distributions
on $(1 \pm 2\dD)d$-regular factorisations 
$(H^c_1,\dots,H^c_{m_c})$ of $H_c \vert (G'_c \cap G_c)$,
which are conditionally independent given $(H'_1,\dots,H'_{m'})$,
and which if $H$ is $(\dD,\dD',dm/n)$-dense 
are $(2\dD,\dD',\eta,d/n)$-quasirandom.

Combining these and using \cref{lem:iterspread},
we obtain a conditionally $C/m$-spread distribution
on $(1 \pm 2\dD)d$-regular factorisations of $H \vert G$,
where $G = \bigcap_c (G'_c \cap G_c)$ has $\mb{P}(G) > 1-n^{-9}$,
which if $H$ is $(\dD,\dD',dm/n)$-dense
is $(2\dD,\dD',\eta,d/n)$-quasirandom.

It remains to show that we can assume $m \ge \sqrt{d}$.
Suppose $m<\sqrt{d}$ and fix an integer $d' \in [\sqrt{d},2\sqrt{d}]$.
Applying \cref{lem:rough} with $m' = md'$ gives
a conditionally $C/m'$-spread distribution
on $(1 \pm 2\dD)d/d'$-regular factorisations 
$(H'_1,\dots,H'_{m'})$ of $H \vert G$.
Then merging groups of $d'$ parts gives
a conditionally $C/m$-spread distribution
on $(1 \pm 2\dD)d$-regular factorisations 
$(H_1,\dots,H_m)$ of $H \vert G$.
Furthermore, if $(H'_1,\dots,H'_{m'})$
is $(2\dD,\dD',\eta,d/d'n)$-quasirandom
then $(H_1,\dots,H_m)$ is $(2\dD,\dD',\eta,d/n)$-quasirandom
(as in the proof of \cref{lem:Hi}).
The lemma follows.
\end{proof}

Henceforth we can assume $dm < 2\log^4 n$ and $m \ge \sqrt{d}$.

As in the proof of \cref{thm:main}, 
we will use the event $G_0$ that $H$ is $(3\dD',\eta dm/n)$-sparse,
which has $\mb{P}(G_0) > 1 - n^{-14}$.
We also use the event $G_1$ that $H$ is $(n^{-.1},6/n)$-sparse,
i.e.\ $|H[V]| \le 6|V|$ whenever $|V| \le 2n^{.9}$.
This has $\mb{P}(G_1) > 1 - n^{-14}$
by \cref{lem:sparse} applied with
$(C,r,\aA,\gG)$ replaced by $(C/6,6dm,n^{-.1},1/dm)$,
noting that $\gG > C(20\aA)^{1/5}$,
as $dm < 2\log^4 n$ and $n>K \gg d \gg C$.
Finally, we also use the event $G_2$ that
we do not have disjoint $S,T \sub V(H)$ with $|T|=20|S| < \sqrt{n}$
and distinct pairs $(s^1_t s^2_t: t \in T)$ such that
each $s^i_t \in S$ and $ts^i_t \in E(H)$.
This has $\mb{P}(G_2) > 1 - n^{-40}$ by \cref{lem:sparse2}, as $n > (Cdm)^8$.

Our good event for the proof of \cref{lem:rough} is $G = G_0 \cap G_1 \cap G_2$.

\subsection{Random Greedy Algorithm}

The input to the algorithm
is a random graph $H$ sampled from a $Cdm/n$-spread distribution
on $(1 \pm \dD)dm$-regular subgraphs of $K_{n,n}$,
assuming that $H$ satisfies the good event $G$ above.

The colouring proceeds in rounds, in each of which
an \emph{active vertex} $v^*$ ranges from $1$ to $n$.
A standard step 
(there will be occasional cleaning and exceptional steps)
in a round will colour some edge $v^* u^*$
so that in expectation each vertex occurs about
once as $v^*$ and once as $u^*$ during the round
(so there will be about $dm/2$ rounds).

We call a round $i$ \emph{early} if $i < (1-\eps) dm/2$;
otherwise we call $i$ \emph{late}.

In each round there is a fixed \emph{active colour} $c^*$,
which ranges cyclically from $1$ to $m$ in successive rounds.
At each step we will increase the partial colouring 
by colouring some new edge, except that occasionally we will 
colour all remaining edges at some particular vertex.

Throughout the algorithm, for any vertex $v$ and colour $c$
we let $\col(v)$ be the number of coloured edges at $v$, 
and let $\col_c(v)$ be the number of these with colour $c$.
\begin{itemize}
\item We say $v$ is \emph{atypical} if at some early round $i$
we have $|\col(v) - 2i| > f(i) dm$, 
where $f(i) = 3\eta (1 + \eps^{-2}/dm)^i$.
\item We say $v$ is \emph{exceptional} if $\ge \tT dm$ edges 
incident to $v$ have been coloured by an exceptional step (defined below).
\item We say $v$ is \emph{$c$-full} if 
 $\col_c(v) \ge 2i/m + \dD d$ at some early round
 or $\col_c(v) \ge (1+2\dD)d-1$ at some late round. 
\item We say $v$ is \emph{$c$-sparse} if in some early round $i$ 
we have $\col_c(v) \le 2i/m - \dD d$. 
\item We say $v$ is \emph{blocked / attacked} 
if there is some colour $c$ such that $\ge d^{.9} m$
neighbours of $v$ have ever been $c$-full / $c$-sparse. 
\item We say $v$ is \emph{blocking / attacking} 
if there are $\ge \tT^2 m$ colours $c$ 
such that $v$ has ever been $c$-full / $c$-sparse.
\item We say that $v$ is \emph{dangerous}
if we colour an edge at $v$ while $v$ is atypical or exceptional
or blocking or attacking or blocked or attacked.
\item We say that $v$ is \emph{unsafe} if it has ever been dangerous
or if we colour an edge at $v$ while $v$ has $>2\eta dm$ unsafe neighbours;
otherwise we say that $v$ is \emph{safe}.
\end{itemize}
We \emph{abort} if $\ge \tT n$ vertices are unsafe. 

At any step,
we let $G$ denote the graph where $V(G)$ is the set of safe vertices
and $E(G)$ is the set of uncoloured edges. 
Also, for any safe $v$ we let
$C_v$ be the set of colours $c \in [m]$ such that $v$ is not $c$-full
and $\le \tT |G(v)|$ neighbours in $G$ of $v$ are $c$-full.

If there are any uncoloured edges on unsafe vertices
then we \emph{clean} them sequentially,
meaning that for each unsafe $v$ in turn, ordered in a \emph{queue},
we colour all remaining uncoloured edges at $v$,
according to a random perfect matching chosen by \cref{lem:spreadmatch}
in the auxiliary balanced bipartite graph $B_v$ defined below.

To maintain the queue, after any step
(standard, cleaning or exceptional)
we consider the \emph{batch} $B$
of all vertices that became unsafe in this step.
We repeatedly increase $B$ by adding any safe $v$
with $\ge 12$ neighbours in $B$ until there are no such $v$.
We let $H_B$ be the graph on $B$ with edges consisting 
of all $uv$ with $u,v \in B$ such that $uv \in H$
or $uw,vw \in H$ for some $w$ that is safe
or in $B' := \{ w \in B: |H[B](w)| < d^{.1} \}$. 
We add $B$ to the end of the queue in a degeneracy order for $H_B$,
so that each has $\le d^{.2}$ earlier neighbours in its batch. 

For each $v \in B$, 
we let $X_v$ consist of all $u$ such that $uv$ is uncoloured.
We let $C_v$ be the set of colours $c$ such that $v$ is not $c$-full
and $\le \tT |X_v|$ many $u \in X_v$ are $c$-full.
We let $Y_v$ be a multiset with $|Y_v|=|X_v|=dm-\col(v)$ 
supported in $C_v$ with multiplicities that differ by $\le 1$.
We let $B_v$ be the bipartite graph on $(X_v,Y_v)$
where $uc$ is an edge whenever $u$ is not $c$-full.
If $|Y_v| \ge |C_v|$ we fix any $Y_v$ as above,
but if $|Y_v|<|C_v|$ we choose $Y_v \sub C_v$ 
uniformly at random and condition on the event $E_v$
that $B_v$ has minimum degree $\ge (1-\tT^{.1})|X_v|$.

For any colour $c$ we let $G_c$ be the subgraph of $G$
with all edges where both ends are not $c$-full.

Now suppose that all unsafe vertices have been cleaned.
If $v^*$ is $c^*$-sparse then 
we perform an \emph{exceptional step}: we choose a uniformly
 random $u^* \in G_{c^*}(v^*)$ and colour $u^*v^*$ by $c^*$.
Otherwise, we perform a \emph{standard step}: 
 we choose a uniformly random $u^* \in G(v^*)$
 and colour $u^*v^*$ by $c$ chosen uniformly at random
 so that $u^*,v^*$ are both not $c$-full. 

We say that a colour $c$ is bad if $\ge \tT^4 n$ vertices 
have ever been $c$-full or $c$-sparse.
We \emph{abort} if any colour is bad.

We update $(v^*,c^*)$ by incrementing $v^*$ by $1$ mod $n$,
where if we had $v^*=n$ then $v^*$ returns to $1$ and then
we also increment $c^*$ by $1$ mod $m$. We repeat the above
colouring procedure until all edges have been coloured.

When all edges have been coloured, the algorithm outputs
the factorisation $(H_1,\dots,H_m)$ of $H$
where each $H_c$ consists of all edges of colour $c$.

\subsection{Basic properties} \label{subsec:basic}

We record some basic properties
of the algorithm in this subsection,
showing in particular that if it does not abort
then it produces a $(1 \pm 2\dD)d$-regular factorisation.

\begin{enumerate}
\item We always have $\col_c(v) \le (1+2\dD)d$ for any $v,c$,
as once $v$ is $c$-full we never again use colour $c$ on edges at $v$.

\item There are three types of step in the algorithm:
 cleaning, exceptional and standard.
Exceptional steps only occur in early rounds 
at safe vertices that are $c$-sparse with $c=c^*$. 

\item Cleaning steps use edges at $\le dm$ neighbours of some vertex,
and exceptional / standard steps use one edge incident to two vertices.
As a vertex can only change status when we colour an edge at it,
in each step the initial batch $B$ of vertices that become unsafe
has size $|B| \le dm$.

\item On the event $G_1$ we have $|H[V]| \le 6|V|$ whenever $|V| \le n^{.9}$,
so when we increase $B$ by repeatedly adding any safe $v$ 
with $\ge 12$ neighbours in $B$ we terminate with $|B| \le 2dm$.

\item 
We have $d^{.1} |B \sm B'| \le \sum_{w \in B \sm B'} |H[B](w)| \le 2|H[B]| \le 12|B|$ 
on $G_1$, so $|B \sm B'| \le 12d^{-.1} |B| \le 24 d^{.9} m$.

\item On the event $G_2$ we claim  
that $|E(H_B)| \le .5d^{.2}|B|$, and so each $v \in B$
has $\le d^{.2}$ earlier neighbours in a degeneracy order of $H_B$.
To see this, we first note that $|H[B]| \le 6|B|$.
Next, the number of $uv \in H_B$ having a common neighbour $w \in B'$
is $\le \sum_{w \in B'} \tbinom{|H[B](w)|}{2} 
\le d^{.1} \sum_{w \in B} |H[B](w)|/2 = d^{.1} |H[B]|$.
Thus if $|E(H_B)| > .5d^{.2}|B|$ we have $>.1d^{.2}|B|$
pairs in $B$ with a common safe neighbour $w$.
However, each such $w$ can be counted $<12$ times by definition of $B$,
so there are $>20|B|$ such $w$, contradicting $G_2$.

\item In particular, for any $v \in B$ we clean $\le d^{.2}$ vertices $u \in H[B](v)$
before $v$, so writing $G'$ for $G$ just before $v$ became unsafe
we have $|G'(v)| - d^{.2} \le |X_v| \le |G'(v)|$.

\item For each $v \in B$, before $v$ become unsafe 
it had $\le 2\eta dm$ unsafe neighbours,
so any $v$ always has $\le 2\eta dm + d^{.2}$ neighbours
that have been cleaned before it.

\item If $v$ is safe it is not blocking,
so $< \tT^2 m$ colours have ever been full at $v$.
If becomes unsafe we colour $\le d^{.2}$ more edges at $v$, 
so $\le d^{.2}$ further colours became full at $v$.
Similarly, each $u \in X_v$ has been full for $< \tT^2 m + d^{.2}$ colours,
so $< \tT m + d^{.2} \tT^{-1} < 2\tT m$ colours can be full 
for $> \tT |X_v|$ many $u \in X_v$, using $m \ge \sqrt{d} \gg \tT^{-1}$.
We deduce that $|C_v| \ge (1-3\tT)m$.

\item If $v$ is safe then it is not blocked, so any colour $c$ has 
$< d^{.9} m$ safe neighbours of $v$ that have ever been $c$-full.
If $v$ becomes unsafe in some batch $B$ then it has $\le d^{.2}$
earlier vertices $u$ in $B$ such that $u,v$ have a common neighbour $w$
that is safe or in $B'$. Each such $u$ receives $<1.1d$ edges of colour $c$
when it is cleaned, so can cause $<1.1d$ neighbours of $v$ to become $c$-full.
Any other $u$ before $v$ in $B$ has $uv \notin E[H_B]$,
so any common neighbour $w$ of $u,v$ is in $B \sm B'$,
which has size $\le 24 d^{.9} m$.
Thus the number of $c$-full $u \in X_v$
is $\le d^{.9} m + 1.1d^{1.2} + 24 d^{.9} m < 25d^{.9} m$, 
as $m \ge \sqrt{d}$.

\item
For any $v$ with $|X_v| > d^{.99} m$, any colour is full for
$< 25d^{.9} m < \tT |X_v|$ vertices $u \in X_v$,
so $C_v$ is the set of colours not full at $v$.
In particular, this holds in all early rounds at any $v$
that is safe and so not atypical.

\item The function $f(i)$ used to define atypical vertices
satisfies $3\eta \le f(i) \le \eta^{.9}$, as $\eta \ll \eps$. 
Thus at any early round $i$, any vertex $v$ that is safe,
and so not atypical, has $|\col(v) - 2i| \le \eta^{.9} dm$, 
so $|G(v)| = dm - 2i \pm \eta^{.9} dm$.

\item If $v$ is cleaned at any stage in the algorithm 
then it receives $|X_v|/|C(v)| \pm 1$
additional edges of each colour, which at an early round $i$
is $d - 2i/m \pm 2\eta^{.9} d$, using $\tT \ll \eta$ and $|C_v| \ge (1-3\tT)m$.

\item If a vertex $v$ becomes $c$-sparse at some early round $i$, 
then while $v$ remains $c$-sparse and safe
at least one edge of colour $c$ is used at $v$ in each round,
so we always have $\col_c(v) \ge 2i/m - \dD d - 2$.

\item If $v$ is cleaned at an early round $i$ 
then the final number of colour $c$ edges at $v$ is 
$\ge 2i/m - \dD d - 2 + d - 2i/m - 2\eta^{.9} d \ge (1-2\dD)d$.
On the other hand, if $v$ is not cleaned at any early round then
at the start of the late rounds we have
$\col_c(v) \ge 2(1-\eps)d/2 - \dD d - 2 \ge (1-2\dD)d$.

\item Thus every colour $H_c$ of the output factorisation is $(1 \pm 2\dD)d$-regular.
\end{enumerate}

We now consider the selection of colours in a cleaning step
according to a spread matching in the auxiliary bipartite graph $B_v$ 
on $(X_v,Y_v)$ with $|Y_v|=|X_v|=dm-\col(v)$,
where $Y_v$ is a multiset supported in $C_v$ 
with multiplicities that differ by $\le 1$,
and $uc$ is an edge whenever $u$ is not $c$-full. 

For any $W \sub G(v)$ and colours $c_W = (c_w: w \in W)$,
let $\mc{A}_W^{c_W}$ be the event 
that $vw$ gets colour $c_w$ for all $w \in W$.

\begin{lem} \label{lem:clean}
Given any history,
the event $E_v$ that $B_v$ has minimum degree
$\ge (1-\tT^{.1})|X_v|$ has $\mb{P}(E_v)>1/2$,
and each event $\mc{A}_W^{c_W} \vert E_v$ has 
probability $\le (L/m)^{|W|}$.
\end{lem}

\begin{proof}
By definition of $C_v$ for $v$ being cleaned,
every $y \in Y_v$ has degree $|B_v(y)| \ge (1-\tT)|X_v|$.
Now consider any $x \in X_v$. As noted above, 
$\le \tT^2 m + d^{.2}$ colours are full at $x$,
so $|B_v(x)| \ge |Y_v| - (\tT^2 m + d^{.2})(|X_v|/|C_v| + 1)$.
If $|Y_v| \ge \tT^{.9} m$ this implies $|B_v(x)| \ge (1-\tT)|X_v|$,
as $m \ge \sqrt{d} \gg 1/\tT$.

On the other hand, if $|Y_v| < \tT^{.9} m < |C_v|$ then
we choose $Y_v \sub C_v$ uniformly at random.
Then each $|B_v(x)|$ is hypergeometric with mean
$\ge (1 - \tT^2 - d^{.2}/m) |X_v|$, so by Chernoff 
$\mb{P}(|B_v(x)| < (1-\tT^{.1})|X_v|) < e^{-\tT^{.3} |X_v|}$, say.
If $|X_v| > \tT^{-.4}$ then by a union bound over $x \in X_v$
we have minimum degree $\ge (1-\tT^{.1})|X_v|$ with failure probability $<1/2$.
On the other hand, if $|X_v| \le \tT^{-.4}$ then we use the simple bound
$\mb{P}(|B_v(x)| < |X_v|) < |X_v| (\tT^2 + d^{.2}/m) < 2\tT^{1.6}$,
so again by a union bound over $x \in X_v$ we have minimum degree 
$\ge (1-\tT^{.1})|X_v|$ with failure probability $<1/2$.

In all cases, $B_v$ has minimum degree $\ge (1-\tT^{.1})|X_v|$,
either deterministically or after conditioning on $E_v$ with $\mb{P}(E_v)>1/2$.
In particular, $B_v$ is $(\tT^{.01},p)$-super-regular
for some $p \in [1-\tT^{.01},1]$, so \cref{lem:spreadmatch}
gives a random perfect matching $M$ that is
conditionally $L/6|X_v|$-spread, say, using $\tT \ll 1/L \ll 1$.
 
To see what this implies for the colouring of edges at $v$,
consider any $\es \ne W \sub X_v$, colours $c_W = (c_w: w \in W)$,
and let $\mc{A}_W^{c_W}$ be the event 
that $vw$ gets colour $c_w$ for all $w \in W$.
There are $\le (|X_v|/|C_v| + 1)^{|W|}$ 
choices of matching in $B_v$ corresponding to this colouring,
each appearing in $M$ with probability $\le (L/2|X_v|)^{|W|}$.
If $|X_v| \ge |C_v|/2$ then  $|X_v|/|C_v| + 1 \le 3|X_v|/|C_v|$,
so $(|X_v|/|C_v| + 1)^{|W|} (L/6|X_v|)^{|W|} \le (L/2m)^{|W|}$. 
We deduce  $\mb{P}(\mc{A}_W^{c_W} \mid E_v) \le (L/m)^{|W|}$.

It remains to consider the case $|X_v|<|C_v|/2$.
Now we recall that $Y_v \sub C_v$ is uniformly random.
We can assume $c_W$ is a set (no repeated colours).
Then $\mb{P}(c_W \sub Y_v) \le ( \frac{|Y_v|-|W|}{|C_v|-|W|} )^{|W|}
\le (2|Y_v|/m)^{|W|}$, as $|Y_v|<m/2<|C_v|$,
so $\mb{P}(\mc{A}_W^{c_W} \mid E_v) \le 2 \mb{P}(c_W \sub Y_v) (L/6|X_v|)^{|W|}
\le (L/m)^{|W|}$, as required.
\end{proof}
 
\subsection{Proof modulo lemmas}

In this subsection we prove \cref{lem:rough},
assuming the following two lemmas that will 
be proved in subsequent subsections.

\begin{lem} \label{lem:danger}
The algorithm aborts with probability $<\sqrt{\tT}$.
\end{lem}

\begin{lem} \label{lem:quasi}
If $H$ is $(\dD,\dD',\eta,dm/n)$-quasirandom
then with failure probability $<n^{-12}$
either the algorithm aborts or
every $H_c$ is $(2\dD,\dD',\eta,d/n)$-quasirandom.
\end{lem}

We also require the following lemma on conditional probabilities 
of certain events in each step of the algorithm.

\begin{lem} \label{lem:step}
Consider any round $i$, any colour $c$ and any $vw \in G$.

1. If there is a standard step with $v^*=v$ 
and $v$ and $w$ are both not $c$-full
then the conditional probability that it uses colour $c$
given that it colours $vw$ is $1/(m \pm 2\tT^2 m)$.

2. If $i$ is early and there is a standard or exceptional step
with $v^*=v$ then the conditional probability that it colours $vw$ 
is $(1 \pm 2\eps^{-1}f(i))/(dm-2i) = (1 \pm \eta^{.8})/(dm-2i) < 2/(\eps dm)$.

3. If there is a standard step with $v^*=v$ and $c \in C_v$
then the conditional probability that it uses colour $c$ is $(1 \pm 2\tT)/m$.
\end{lem}

\begin{proof}
For (1), recall that in a standard step
we choose a uniformly random $u^* \in G(v^*)$
and colour $u^*v^*$ by $c$ chosen uniformly at random
so that $u^*,v^*$ are both not $c$-full. 
We note that $v$ is safe, so not blocking,
so full for $< \tT^2 m$ colours.
Similarly, any uncoloured edge $uv$ has $u$ safe,
so not blocking, so full for $< \tT^2 m$ colours.
This implies (1), as we always choose
a colour from $m \pm 2\tT^2 m$ options.

For (2), note that $v$ is safe, so not atypical,
so $|\col(v) - 2i| \le f(i) dm \le \eta^{.9} dm$, as $i$ is early.
For a standard step we thus choose $u^*$ from
$|G(v)|=dm-\col(v) = dm-2i \pm f(i) dm$ options.
For an exceptional step, recall that we have some fixed colour $c^*=v$
and colour $v u$ where $u$ is a uniformly random neighbour of $v$
in the subgraph $G_c$ of $G$ of edges
where both ends are not $c$-full.
As $v$ is safe it is not blocked,
so $< d^{.9} m$ neighbours of $v$ are $c$-full.
Thus we choose $u^*$ from
$dm-2i \pm (f(i)+d^{-.1}) dm$ options.
The estimates in (2) follow as
$1/d \ll \tT \ll \eta \le f(i) \le \eta^{.9} \ll \eps \le 1-2i/dm$.

For (3), by definition of $C_v$ we choose $u^*$ not $c$-full 
with probability $1 \pm \tT$.
By (1) we then use $c$ with probability $1/(m \pm 2\tT^2 m)$, so (3) follows.
\end{proof}

\begin{proof}[Proof of \cref{lem:rough}] 
Recall that we have conditioned $H$ on a good event $G = G_0 \cap G_1$.
We now define two good events for the algorithm applied to $H \vert G$.
We let $\mc{G}$ be the good event that the algorithm does not abort,
and so outputs $(H_1,\dots,H_m)$ sampled from some distribution
on $(1 \pm 2\dD)d$-regular factorisations of $H$.
Then $\mb{P}(\mc{G}) > 1-\sqrt{\tT}$ by \cref{lem:danger}.
We let $\mc{G}'$ be the good event that
if $H$ is $(\dD,\dD',dm/n)$-dense
(and so $H \vert G$ is $(2\dD,\dD',\eta,d/n)$-quasirandom)
then every $H_c$ is $(2\dD,\dD',\eta,d/n)$-quasirandom.
Then $\mb{P}(\mc{G}') > 1-n^{-12}$ by  \cref{lem:quasi}.

Consider $S = (S_1,\dots,S_m)$ for some disjoint $S_c \sub E(K_{n,n})$.
Let $\mc{A}_S$ be the event $\bigcap_{c=1}^m \{ S_c \sub H_c \}$
and let $\mc{A}'_S$ be the event $\bigcap_{c=1}^m \{ S_c \sub H \}$.
Write $s = \sum_c |S_c|$. It suffices to show 
$\mb{P}(\mc{A}_S \cap \mc{G} \mid \mc{A}'_S) \le (C/2m)^s$.
Indeed, this implies \cref{lem:rough}.1, as we have
$\mb{P}(\mc{A}_S \mid \mc{G} \cap \mc{A}'_S) \le (C/m)^s$.
It also implies \cref{lem:rough}.2,
as if $H$ is $(\dD,\dD',dm/n)$-dense then
$\mb{P}(\mc{A}_S \mid \mc{G} \cap \mc{G}' \cap \mc{A}'_S) \le (C/m)^s$.
Note that we are only using \cref{lem:quasi} for the proof of \cref{lem:rough}.2,
so that we can later use \cref{lem:rough}.1 in the proof of \cref{lem:quasi}.

We partition $\mc{A}_S$ into events 
$(\mc{A}_{S,X,T,U,z,y}: T,U \sub S, X \sub V(S), z \in ([2] \times [d])^U, y \in [3]^s)$,
where $X$ is the set of vertices in $V(S)$ that are cleaned,
$T$ is the set of edges in $S$ coloured in standard steps,
$U$ is the set of edges in $S$ coloured in exceptional steps,
$z_e$ for each $e=uv \in U$ specifies the bijection between
$uv$ and $u^* v^*$ and how many exceptional steps with $c^*=c$ where $e \in S_c$ 
occur at $v^*$ up to and including the step when $e$ is coloured,
and $y_j$ for each $j \in [s]$ specifies whether the $j$th edge of $S$
to be coloured by the process is coloured by a standard step, an exceptional step,
or by cleaning (considering edges coloured
during any single cleaning step in an arbitrary order).

Taking a union bound, it suffices to show each
$\mb{P}(\mc{A}_{S,X,T,U,z,y} \cap \mc{G} \mid \mc{A}'_S) < (2d)^{-|U|} (C/100m)^s$.

Fix any $\mc{A} = \mc{A}_{S,X,T,U,z,y}$.
We imagine a \emph{monitor} for $\mc{A}$
that is sometimes asleep and not observing the process: 
it is awake exactly when a standard step chooses an uncoloured edge in $S$
or for some vertex $v$ incident to any uncoloured edge in $S$
we clean $v$ or an exceptional step occurs at $v$.
Each time the monitor wakes it rejects $\mc{A}$
if it sees an outcome inconsistent with $\mc{A}$, 
otherwise it goes back to sleep.
At the end of the process, $\mc{A}$ is accepted if it has not been rejected

We will bound the acceptance probability by a product over each step
when the monitor wakes of a bound on the conditional probability 
that it does not reject at this step.
   
When a standard step chooses an uncoloured edge $e$ in $S$
the monitor will reject unless 
this is consistent with $y$ and $e \in T$, when
it does not reject with probability $<2/m$ by \cref{lem:step}.1.

When we clean $v$ incident to any uncoloured edge in $S$
the monitor will reject unless $v \in X$,
this is consistent with $y$,
and the set $S_v$ of such uncoloured edges consists exactly 
of those edges in $S \sm (T \cup U)$ containing $v$, when
by \cref{lem:clean} it does not reject with probability $<(L/m)^{|S_v|}$.
 
When an exceptional step occurs at some $v$ 
incident to any uncoloured edge in $S$,
the monitor will reject 
if it colours some edge $e$ of $S$ unless $e \in U$ 
and this is consistent with $y$ and $z$,
when by \cref{lem:step}.2 it does not reject 
with probability $< 2/(\eps dm)$.

Crucially, the above estimates hold for any history and only depend
on the the sequence $y$ of types of events waking the monitor -
they do not depend on which edges are involved at each step.

We can apply these estimates inductively for $j \in [s]$
(we do not consider any $j$ with $y_j$ corresponding to an incomplete cleaning step).
Writing $j=j_1+j_2+j_3$, where the entries of $y_{\le j}$ have
$j_1$ corresponding to standard steps, $j_2$ to exceptional steps and $j_3$ to cleaning,
we see that the monitor has not yet rejected with probability
$< (2/m)^{j_1} (2/(\eps dm))^{j_2} (L/m)^{j_3}$.

Taking $C > 10^3 L\eps^{-1}$ we deduce 
$\mb{P}(\mc{A}_{S,X,T,U,z,y} \cap \mc{G} \mid \mc{A}'_S) < (2d)^{-|U|} (C/100m)^s$.
\end{proof}

\subsection{Quasirandomness}

Here we show that the algorithm maintains quasirandomness.

\begin{proof}[Proof of \cref{lem:quasi}]
Suppose $H$ is $(\dD,\dD',\eta,dm/n)$-quasirandom.
We assume that the algorithm does not abort and bound the probability that 
some $H_c$ is not $(2\dD,\dD',\eta,d/n)$-quasirandom.

Firstly, as shown in \cref{subsec:basic}, 
each $H_c$ is $(1 \pm 2\dD)d$-regular.

Secondly, by \cref{lem:rough}.1
the algorithm produces a factorisation $(H_1,\dots,H_m)$ of $H$
that is $2C^2 d/n$-spread conditional on $\mc{G}$.
We apply \cref{lem:sparse} with
$(2C^2,d,3\dD',\eta)$ in place of $(C,r,\aA,\gG)$,
noting that $\eta > 2C^2 (60\dD')^{1/5}$ and $d > 5/\eta$,
as $d \gg 1/\dD' \gg 1/\eta \gg C$.
Thus all $H_c$ are $(3\dD',\eta d/n)$-sparse
with failure probability $<n^{-13}$.

Thirdly, it remains to show density.
Consider any $V_i \sub U_i$ for $i=1,2$ with $|V_1|,|V_2| \ge \dD' n$.
By density of $H$ we have $|H[V_1,V_2]| \ge (1-\dD) |V_1||V_2| dm/n$.
We need to bound the failure probability of the event 
that all $|H_c[V_1,V_2]| \ge (1-2\dD) |V_1||V_2| d/n$.

As the algorithm did not abort, there are $< \tT n$ unsafe vertices.
Thus $< \tT n dm < \dD' |H[V_1,V_2]|$ edges are incident 
to unsafe vertices, using $\tT \ll \dD'$ and $|H[V_1,V_2]| \ge .9(\dD')^2 dmn$.
Any safe $v$ is not exceptional, so $< \tT dm$ edges at $v$ are coloured in
exceptional steps; this accounts for $< n \tT dm < \dD' |H[V_1,V_2]|$ edges again.
Also, there is no bad colour, so any colour $c$ has $< \tT^4 n$ vertices 
that are $c$-full, so $< \tT^4 n (|V_1|+|V_2|) < \dD' |H[V_1,V_2]|$ 
edges are incident to vertices that are $c$-full.
 
Fix $c$ and let $H'$ be the (random) set of edges in $H[V_1,V_2]$
coloured at standard steps where both $v^*,u^*$ are not $c$-full.
The above estimates show $|H'| \ge (1-3\dD')|H[V_1,V_2]|$. 
Then $|H_c[V_1,V_2]| \ge X = \sum_j X_j$, where $X_j$ is the indicator 
that the $j$th coloured edge of $H'$ receives colour $c$.
At each standard step, in some round $i$ with $v^*=v$,
the history $\mc{F}_{i,v}$ and the choice of $u^*$
determines whether $v^*u^*$ is in $H'$ 
and so could be counted by some $X_j$, then
$\mb{E}(X_j \mid u^*, \mc{F}_{i,v}) = 1/(m \pm 2\tT^2 m)$
by \cref{lem:step}.1.
 Thus we can couple $H'$ to a binomial variable with mean
$>(1-3\dD')|H[V_1,V_2]| \cdot 1/(m \pm 2\tT^2 m) > (1-1.1\dD) |V_1||V_2| d/n$.
Using Chernoff and $1/d \ll \tT \ll \dD'$ we deduce
$\mb{P}( |H_c[V_1,V_2]| < (1-2\dD) |V_1||V_2| d/n ) < \tT^n$, say.
Taking a union bound over $m<n$ colours
and $<4^n$ choices of $V_1,V_2$, the lemma follows.  
\end{proof}  

\subsection{Analysis of algorithm}

To complete the proof of \cref{lem:rough},
and so of \cref{thm:main+},
it remains to prove \cref{lem:danger},
i.e.\ that the algorithm aborts with probability $<\sqrt{\tT}$.

\begin{proof}[Proof of \cref{lem:danger}]
The input to the algorithm
is a random graph $H$ sampled from a $Cdm/n$-spread distribution
on $(1 \pm \dD)dm$-regular subgraphs of $K_{n,n}$.

We assume that $H$ satisfies the good event $G$ above,
so in particular $H$ is $(3\dD',\eta dm/n)$-sparse.
Let $U$ be the set of unsafe vertices and
$D$ be the set of vertices that have ever been dangerous.
We claim that if $|U| \ge \tT n$ then $|D| \ge \tT n/2$.
Indeed, suppose not and consider $U' \sub U$ with $|U'|=\tT n$
consisting of $D$ and the first $\tT n - |D|$ vertices 
added to $U$ due to having $> 2\eta dm$ unsafe neighbours.
Then $2\eta dm (\tT n - |D|)  < |H[U']| \le |U'| \eta dm$,
so $2(\tT n - |D|) < \tT n$, so the claim holds.

We define a stopping time $\tau$ 
as the first step when we first have
$\ge \tT n$ unsafe vertices or some bad colour, 
or $\infty$ if there is no such step.
We want to show $\mb{P}(\tau<\infty)<\sqrt{\tT}$.
To estimate $\mb{P}(\tau<\infty)$, we will use the fact
that if $\tau=s$ then the algorithm did not abort 
at any step before $s$.

We claim that $\{\tau<\infty\} \sub \bigcup_{i=1}^4 E_i$, where
\begin{itemize}
\item $E_1$ is the event that there are
$\ge \tT^2 n$ vertices $v$ each with
$\ge \tT dm$ edges coloured by an exceptional step,
\item $E_2$ is the event that there are
$\ge \tT^2 n$ vertices $v$ each with
some early round $i$ failing $\col(v) = 2i \pm f(i) dm$,
\item $E_3$ is the event that there is a bad colour,
i.e.\ some $c$ with $> \tT^4 n$ vertices 
that have ever been $c$-full or $c$-sparse.
\item $E_4$ is the event that there are
$\ge \tT^2 n$ vertices $v$ that are blocked or attacked.
\end{itemize}

To see this claim, we first note that $E_3$ covers
the event of aborting due to a bad colour,
so it remains to consider the event of aborting
due to having $>\tT n$ unsafe vertices,
and so $\ge \tT n/2$ vertices that have ever been dangerous,
and so atypical or exceptional
or blocking or attacking or blocked or attacked.
By definitions of $E_1$, $E_2$ and $E_4$ this requires $\ge \tT n/4$
vertices that are blocking or attacking.
However, if we have $>.1\tT n$ vertices that are blocking / attacking
then there are $>.1\tT n \cdot \tT^2 m$ pairs $(v,c)$
such that $v$ has ever been $c$-full / $c$-sparse;
some $c$ must occur in $>.1\tT n \cdot \tT^2 > \tT^4 n$
such pairs and so is a bad colour. The claim follows.
 
We start by bounding $\mb{P}(E_1)$.
We note that any safe vertex $v$ is not attacking,
so has $< \tT^2 m$ colours $c$ such that 
$v$ has ever been $c$-sparse, so there are 
$< \tT^2 dm$ exceptional steps with $v^*=v$.
The same holds for any safe neighbour of $v$.
If $v$ is exceptional then we have 
$> .9\tT dm$ exceptional steps with $u^*=v$.
For any exceptional step with $v^*=w$ at any neighbour $w$ of $v$, 
which must occur in an early round, by \cref{lem:step}.2
the conditional probability of choosing $u^*=v$ is $< 2/(\eps dm)$.
The total number of such steps is dominated by a binomial with mean
$< dm \cdot \tT^2 dm \cdot 2/(\eps dm) < \tT^{1.9} dm$.
as $\tT \ll \eps$. By Chernoff, the probability 
that $v$ becomes exceptional is $< \tT^3$, say. 
By Markov we deduce $\mb{P}(E_1)<\tT$.

Next we bound $\mb{P}(E_2)$.
Consider any $v$ that is atypical but not exceptional,
so that $<\tT dm$ edges at $v$ are coloured by an exceptional step.
While $v$ was safe, $< 2\eta dm$ edges at $v$ were coloured by cleaning.
Suppose $v$ became atypical in some (early) round $i'$ and 
let $S'(i',v)$ count standard steps with $u^*=v$ while $v$ was safe.
In each previous round we had one standard 
or exceptional step with $v^*=v$,
so $|S'(i',v) - i'| > (f(i')-2\eta)dm$.
We recall $f(i) = 3\eta (1 + \eps^{-2}/dm)^i$
and note that $\sum_{i \le i'} f(i) < 2\eps^2 dm f(i')$.
In any early round $i$ before $v$ becomes atypical, writing $G^i=G$,
there are $|G^i(v)|=dm-\col(v)=dm-2i \pm f(i)dm$ safe neighbours $w$ of $v$.
For each $w \in G^i(v)$, by \cref{lem:step}.2 if there is a standard step with $v^*=w$ 
then it chooses $u^*=v$ with conditional probability 
$(1 \pm 2f(i)\eps^{-1})/(dm-2i)$.

Consider independent Bernoulli variables $X^1_{i,w}$, $X^0_{i,w}$ 
for $w \in G^i(v)$ with each $\mb{E}X^1_{i,w} = 1/(dm-2i)$
and $\mb{E}X^0_{i,w} = 3f(i)\eps^{-1}/(dm-2i)$.
Let $X^j = \sum_{i \le i'} \sum_{w \in G^i(v)} X^j_{i,w}$ for $j=0,1$.
Let $X^*$ be the sum of $X^1_{i,w}$ over all $(i,w)$ 
where there is an exceptional step at $w$ in round $i \le i'$.
Then we can couple $S'(i',v) \in X^1 - X^* \pm X^0$:
for each standard step in round $i$ with $v^* = w \in G^i(v)$
we can construct the indicator of the event $\{u^*=v\}$
by starting with the term $X^1_{i,w}$ in $X^1 - X^*$  
and possibly correcting it via the term $\pm X^0_{i,w}$,
subsampled to give the correct conditional probability given the history.

We note that
$\mb{E}X^1 = \sum_{i \le i'} |G^i(v)|/(dm-2i) 
 = i' \pm \sum_{i \le i'} f(i)\eps^{-1} = i' \pm 2\eps f(i') dm$ 
and $\mb{E}X^0 = \sum_{i \le i'} |G^i(v)| \cdot 3f(i)\eps^{-1}/(dm-2i)
 < \sum_{i \le i'} 4f(i)\eps^{-1} < 8\eps f(i') dm$, 
as $|G^i(v)| < 1.1(dm-2i)$ for early $i$.
Also, as each $w \in G^i(v)$ is not exceptional,
$<\tT dm$ edges at $w$ are coloured by an exceptional step,
so $\mb{E}X^* < dm \cdot \tT dm \cdot 2/ \eps dm < \tT^{.9} dm$, as $\tT \ll \eps$.
We can cover the event that $v$ becomes atypical by the events
$\{ |X^1-i'| > .9f(i')dm \}$ or
$\{ |X^0| > 16\eps f(i')dm\}$ or
$\{ |X^*| > 2\tT^{.9} dm\}$.
Thus by Chernoff, $v$ becomes atypical with probability $<\tT^3$, say.
By Markov we deduce $\mb{P}(E_2)<\tT$.

To bound $\mb{P}(E_3)$ and $\mb{P}(E_4)$ we first set up couplings
so that we can use Chernoff bounds to estimate the probability 
of any vertex $v$ becoming $c$-full or $c$-sparse for some colour $c$;
we will bound this probability by $e^{-\sqrt{d}}$.

We consider any step before the stopping time $\tau$
at which $v$ first becomes $c$-full or $c$-sparse.
Suppose first that this is at an early round $i$
(which must be the case if $v$ becomes $c$-sparse).
While $v$ is safe it is not atypical or exceptional,
so there are $2i \pm \eta^{.9} dm$ coloured edges at $v$,
of which $<2\eta dm$ are coloured by cleaning
and $<\tT dm$ are coloured by an exceptional step.
It it becomes unsafe then $\le d^{.2}$ further edges at $v$
are coloured before it is cleaned.
Thus there are $2i \pm 2\eta^{.9} dm$ 
edges at $v$ coloured by a standard step.

While $v$ is safe it is not blocked,
so has $<d^{.9} m$ neighbours that are $c$-full.
Let $D_i(v,c)$ count edges at $v$ coloured at standard steps
where $v^*$ is a $c$-full neighbour of $v$.
We fix independent Bernoulli variables $D(w,j,c)$ 
for each $w \in H(v)$, $j \in [dm]$
with each $\mb{E}D(e,c) = 2/\eps dm$. 
We couple $D_i(v,c) \le D^+_i(v,c)$ where at each
standard step with $v^*=w$ some $c$-full neighbour of $v$
we couple choosing $u^*=v$ below 
some new $D(w,j,c)$ added to $D^+_i(v,c)$.
We add further such independent Bernoulli's so that
$D^+_i(v,c)$ is a sum of exactly $d^{.9} m \cdot dm$ 
such variables, so is binomial with mean $< 2\eps^{-1} d^{.9} m$.
By Chernoff, we have $D^+_i(v,c) < d^{.91} m$
with failure probability $<.1e^{-\sqrt{d}}$, say.

Let $S_i(v,c)$ count edges coloured at $v$ by standard steps
at which $v^*$ is not $c$-full. At any such step 
before $v$ is $c$-full or $c$-sparse,
the conditional probability 
of using colour $c$ is $(1 \pm 2\tT)/m$;
indeed this holds by \cref{lem:step}.1 if $u^*=v$,
or by \cref{lem:step}.3 if $v^*=v$
(noting that $c \in C_{v^*}$ as $i$ is early).

We fix independent Bernoulli variables $X(e,c),Y(e,c)$ for each $e \in E(H)$
with each $\mb{E}X(e,c) = 1/m$ and $\mb{E}Y(e,c) = 3\tT/m$. 
We can couple $S_i(v,c) \in X_i(v,c) \pm Y_i(v,c)$, 
where $X_0(v,c)=Y_0(v,c)=0$
and each time we colour some edge $e$ at $v$ by a standard step
at which $v^*$ is not $c$-full
we add $X(e,c)$ to $X_i(v,c)$ and $Y(e,c)$ to $Y_i(v,c)$.

If $D^+_i(v,c) < d^{.91} m$ then
$X_i(v,c) \in X^\pm_i(v,c) := [X^-_i(v,c),X^+_i(v,c)]$,
where $X^\pm_i(v,c)$ sum the first $2i \pm 2\eta^{.9} dm$
values of $X(e,c)$ at $v$, using new independent Bernoulli variables
if there are not enough such edges at $v$, and similarly for $Y$.
We note that although the choice of each $e$ depends on the history of the process,
these choices do not affect the distribution of $X^\pm_i(v,c)$ and $Y^\pm_i(v,c)$,
which are binomial with
$\mb{E}X^\pm_i(v,c) = 2i/m \pm 2\eta^{.9}d$ and $\mb{E}Y^+_i(v,c) \le 3\tT d$.

If $v$ becomes $c$-sparse in round $i$
then we have $X_i(v,c) - Y_i(v,c) < 2i/m - \dD d$.
Consider $i' = \bfl{i/\eps dm} \eps dm$.
Then $0 \le i-i' \le \eps dm$ and at step $i'$ we had 
$X^-_{i'}(v,c) - Y^+_{i'+\eps dm}(v,c) < 2i/m - \dD d < 2i'/m - .9\dD d$,
so $Y^+_{i'+\eps dm}(v,c) > 6\tT d$ or $X^-_{i'}(v,c) < 2i'/m - .8\dD d$.
By Chernoff, taking a union bound over $< \eps^{-1}$ such $i'$,
we can bound the probability that $v$ becomes $c$-sparse by $.1e^{-\sqrt{d}}$, 
say, as $1/d \ll \tT \ll \eta \ll \eps \ll \dD$.

Next we consider the event that $v$ 
becomes $c$-full before it becomes $c$-sparse
in some early round $i$.
As $v$ was not $c$-sparse, edges at $v$ coloured $c$
by exceptional steps must have $u^*=v$ and $v^*=w \in G(v)$.
While $v$ is safe it is not attacked, 
so $< d^{.9} m$ such $w$ can be $c$-sparse
and so have up to $d$ exceptional steps with $v^*=w$ and $c^*=c$.
By \cref{lem:step}.2 such a step chooses $u^*=v$
with conditional probability $< 2/(\eps dm)$.
We couple the number of such steps as  $E_i(v,c) < E^+_i(v,c)$,
where $E^+_0(v,c)=0$, each time we colour some edge $e$ at $v$ 
by colour $c$ using an exceptional step we add $E(e,j,c)$ to $E^+_i(v,c)$,
where $(E(e,j,c))$ are independent Bernoulli's with $\mb{E}E(e,j,c) = 2/(\eps dm)$,
and we add new independent Bernoulli's so that $E^+_i(v,c)$
is a sum of exactly $d^{1.9} m$ such variables, 
so is binomially distributed no matter which edges $e$ are chosen,
with $\mb{E}E^+_i(v,c) = 2\eps^{-1} d^{.9}$.
As $1/d \ll \tT \ll \eps$, by Chernoff
we have $\mb{P}(E^+_i(v,c) > \tT d) < .1e^{-\sqrt{d}}$, say. 

Next, by \cref{lem:clean} each time we clean a neighbour of $v$
we use colour $c$ at $v$ with conditional probability $\le L/m$.
We couple the number of such steps as $C_i(v,c) < C^+_i(v,c)$,
where $C^+_0(v,c)=0$ and each time we clean some neighbour $w$ of $v$ 
we add $C_w(v,c)$ to $C^+_i(v,c)$, where $(C_w(v,c): w \in H(v))$ 
are independent Bernoulli's with $\mb{E}C_w(v,c) = L/m$,
and we add new independent Bernoulli's so that $C^+_i(v,c)$
is a sum of exactly $3\eta dm$ such variables, so is binomially 
distributed with $\mb{E} C^+_i(v,c) = 3L\eta d$
 no matter which neighbours of $v$ are cleaned.
As $1/d \ll \tT \ll \eta \ll 1/L$, by Chernoff 
we have $\mb{P}(C^+_i(v,c) > \sqrt{\eta} d) < .1e^{-\sqrt{d}}$, say. 

For any $v$ with $D_i(v,c) < d^{.91} m$,
$E_i(v,c) \le \tT d$ and $C_i(v,c) \le \sqrt{\eta} d$,
if $v$ becomes $c$-full in round $i$
then we have $X_i(v,c) + Y_i(v,c) > 2i/m + .9\dD d$,
so at step $i' = \bcl{i/\eps dm} \eps dm$ we will have 
$X^+_{i'}(v,c) + Y^+_{i'}(v,c) > 2i'/m + .8\dD d$, which by Chernoff
occurs for some such $i'$ with probability $<.1e^{-\sqrt{d}}$, say.

Now consider the event that $v$ does not become $c$-sparse or $c$-full 
in some early round, then becomes $c$-full in some late round.
At the start of the late rounds $\le (1-\eps+\dD)d$ 
edges incident to $v$ have received colour $c$,
so $\ge \dD d$ such edges will receive colour $c$ during the late rounds.
While $v$ was safe it was not atypical, so at the start of the late rounds
$< (\eps + \eta^{.8}) dm < 2\eps dm$ edges at $v$ are uncoloured.
All such edges are coloured by standard steps or cleaning.
For cleaning steps we continue to update $C^+_i(v,c)$ as above.
We bound such edges in standard steps as $L_i(v,c) \le L^+_i(v,c)$,
where similarly to above using \cref{lem:step}.1
we update $L^+_i(v,c)$ by adding independent
Bernoulli's $L(e,c)$ with $\mb{E}L(e,c) = 2/m$, adding new variables
so that we have exactly $2\eps dm$ such variables,
so $L^+_i(v,c)$ is binomial with mean $4\eps d$.
Thus by Chernoff the event considered by this paragraph
has probability $<.1e^{-\sqrt{d}}$, say.

Combining the above estimates, we can bound the probability
of $v$ becoming $c$-full or $c$-sparse by $e^{-\sqrt{d}}$, as desired.

Next we will bound $\mb{P}(E_3)$.
We fix any colour $c$ and bound the number of $c$-full vertices
by $Z = \sum_v Z(v,c)$, where each $Z(v,c)$
is the indicator of the event that
\[ \Ss^+_i(v,c) := X^+_i(v,c) + Y^+_i(v,c) + E^+_i(v,c) + L^+_i(v,c) + C_i(v,c) + D_i(v,c) \]
is $> 2i/m + \dD d$ in some early round $i$
or $\ge (1 + 2\dD)d-1$ in some late round.
(Note that here we consider $C,D$ rather than $C^+,D^+$.)
Then $\mb{E}Z < e^{-\sqrt{d}} n$ by the above estimates.

We note that if $W$ is an independent set
then $(X^+_i(v,c): v \in W)$ are independent,
and similarly replacing $X$ by $Y$, $E$ or $L$ (but not by $C$ or $D$).
Indeed, each $X^+_i(v,c)$ is a sum of $2i + 2\eta^{.9} dm$
independent iid Bernoulli's $X(e,c)$, 
where the choices of $e$ for each $v \in W$ 
depend on the history of the process and each other, 
but the choice of $e$ does not depend on $X(e,c)$ and
the sets of all possible $e$ are disjoint for distinct $v \in W$.

We write $Z \le Z_0+Z_1+Z_2$, where 
$Z_1$ counts $v$ with $C_i(v,c) > .2\dD d$,
$Z_2$ counts $v$ with $D_i(v,c) > d^{.91} m$,
and $Z_0$ counts $v$ where 
$T_i(v,c) := X^+_i(v,c) + Y^+_i(v,c) + E^+_i(v,c) + L^+_i(v,c)$ 
is $> 2i/m + .4\dD d$ in some early round $i$
or $\ge (1 + 1.4\dD)d-1$ in some late round.
The above estimates show that all
$\mb{E}Z_j < e^{-\sqrt{d}} n$.
Furthermore, $T_i(v,c)$ and $T_i(v',c)$
are independent whenever $v,v'$ are non-adjacent, so $\var(Z_0) \le (dm+1)n$.
By Chebyshev, as $n \ge K \gg 1/\tT$ and $dm \le 2\log^4 n$ 
we have $\mb{P}(Z_0 > .2\tT^4 n) \le \var(Z_0)/(.1\tT^4 n)^2 < n^{-.9}$, say.

For $Z_1$ and $Z_2$, an approach via coupling would be more complicated
due to the use of spread perfect matchings in cleaning,
so we will adopt a simpler martingale argument. Any any step in the algorithm 
let $C(v)$ denote the number of neighbours of $v$ that have been cleaned, 
and $C(v,c)$ denote the corresponding value of $C_i(v,c)$.
We consider a \emph{predictor} $P=\sum_v P_v$ for $Z_1$,
where each $P_v = \min \{ 1, \exp[ C(v,c) - 4C(v)L/m - .1\dD d ] \}$.
The starting value of $P$ is $ne^{-.1\dD d}$.
As $C(v) < 2\eta dm + d^{1.2}$ and $\eta \ll 1/C \ll \dD \ll 1/L$,
if $C(v,c) > .2\dD d$ then $P_v=1$, so $Z_1 \le P$.
Now we claim that $P$ is a supermartingale.
To see this, we consider any cleaning step,
let $Q = \sum_v Q_v$ denote the updated value of $P$,
and show $\mb{E}'[Q_v] \le P_v$ for each $v$, where $\mb{E}'$ 
denotes conditional expectation given the history.
Indeed, we can assume $P_v<1$ and that the
cleaned vertex colours an edge at $v$.
Then we increase $C(v)$ by $1$, and $C(v,c)$ either increases by $1$
with probability $\le L/m$ or is unchanged otherwise, so as $m \ge \sqrt{d} \gg L$
we have $\mb{E}'[Q_v]/P_v \le e^{-4L/m}( (L/m) \cdot e + (1-L/m) \cdot 1) <1$.
Furthermore, $|Q - P| \le \DD(H) = dm$,
so the variance proxy is $<ndm(dm)^2 < n(\log n)^{20}$, say.
By Azuma's inequality (\cref{lem:azuma}),
we conclude $\mb{P}(Z_1 > .2\tT^4 n) < e^{-\sqrt{n}}$, say.

Similarly, for $Z_2$ we consider a predictor, abusively also denoted $P=\sum_v P_v$,
with each $P_v = \min \{ 1, \exp[ D(v,c) - 8D(v)/\eps dm - d^{.9} m] \}$,
where $D(v)$ denotes the number of standard steps 
with $v^*=w$ a $c$-full neighbour of $v$
and $D(v,c)$ the number of these with $u^*=v$.
We note that $Z_2 \le P$, as $D(v) \le d^{1.9} m$,
so if $D(v,c) > d^{.91} m$ then $P_v=1$.
The starting value is $ne^{- d^{.9} m}$,
and $P$ is a supermartingale, as each
$\mb{E}'[Q_v]/P_v \le e^{-8/\eps dm}( (2/\eps dm) \cdot e + (1-2/\eps dm) \cdot 1) <1$.
The variance proxy is $< n(\log n)^{20}$, so by Azuma
$\mb{P}(Z_2 > .2\tT^4 n) < e^{-\sqrt{n}}$, say.

Taking a union bound over colours $c \in [m]$,
we deduce $\mb{P}(E_3) < 2mn^{-.9} < n^{-.8}$, as $m < \log^4 n$.

It remains to bound $\mb{P}(E_4)$.
First we consider the probability that any vertex $v$
is blocked / attacked, i.e.\ has some colour $c$ such that 
$\ge d^{.9} m$ neighbours of $v$ are $c$-full / $c$-sparse.
We fix $c$ then take a union bound later. 
Recalling that $|H[V]| \le 12|V|$ whenever $|V| \le dm$,
by Tur\'an's Theorem we can choose an $H$-independent set $W$ of say 
$d^{.8} m$ neighbours of $v$ that are $c$-full / $c$-sparse. 
If they are $c$-sparse then for each $w \in W$ we have 
$X^-_{i_w}(w,c)-Y^+_{i_w}(w,c) \le 2i_w/m - \dD d$ in some early round $i_w$;
we denote this event by $A_1(v,W)$.
If they are $c$-full then for each $w \in W$
we have $\Ss^+_{i_w}(w,c) > 2i_w/m + \dD d$ in some early round $i_w$
or $\Ss^+_{i_w}(w,c) \ge (1 + 2\dD)d-1$ in some late round;
we let $A_2(v,W)$ be the event that $T_{i_w}(w,c)$ 
is $> 2i_w/m + .5\dD d$ in some early round $i_w$
or $\ge (1 + 1.5\dD)d-1$ in some late round, for all $w \in W$;
we let $A_3(v,W)$ be the event that 
some $C_{i_w}(w,c)>.2\dD d$ for all $w \in W$;
we let $A_4(v,W)$ be the event that 
some $D_{i_w}(w,c)>d^{1.9}m$ for all $w \in W$.

We fix $c$ and bound the number of vertices blocked / attacked for $c$
by $Z' = \sum_{i=1}^4 Z'_i$, where each $Z'_i = \sum_v Z'_{i,v}$ 
and $Z'_{i,v}$ is the indicator that some $A_i(v,W)$ occurs.
We can apply the second moment method as above to $Z'_1$ and $Z'_2$.
Indeed, by independence we have 
$\mb{P}(A_i(v,W)) < e^{-\sqrt{d}|W|}$ for $i=1,2$.
Taking a union bound over $<m2^{dm}$ choices of $c$ and $W$,
we thus have $\mb{E}(Z'_1+Z'_2) < ne^{-d^{1.2}m}$, say.
Furthermore, $Z'_{i,v}$ and $Z'_{i,v'}$ for $i=1,2$ are independent
whenever $v$ and $v'$ are at distance $\ge 4$ in $H$,
as then they depend on disjoint sets of independent variables.
For $i=1,2$ we deduce $\var(Z'_i) \le (dm+1)^4 n$, so by Chebyshev,
$\mb{P}(Z'_i > .2\tT^2 n) \le \var(Z'_i)/(.1\tT^2 n)^2 < n^{-.9}$, say.

For $Z'_3$, we again handle cleaning via a martingale.
We consider a predictor $P'=\sum_v P'_v$ for $Z'_3$, 
with each $P'_v = \min \{ 1, \sum_W P'_{v,W} \}$,
where $P'_{v,W} = \exp[ \sum_{w \in W} ( C(w,c) - 4C(w)L/m - .1\dD d ) ]$
and $W$ ranges over independent sets of $d^{.8} m$ neighbours of $v$.
The starting value of $P'$ is 
$< n \tbinom{dm}{d^{.8} m} (e^{- .1\dD d })^{d^{.8} m} < ne^{-d^{1.7} m}$, say.
As each $C(w) < 2\eta dm + d^{1.2}$ and $\eta \ll 1/C \ll \dD \ll 1/L$,
if there is some such $W$ such that
$C(w,c) > .2\dD d$ for all $w \in W$ then $P'_v=1$, so $Z'_3 \le P'$.
Now we claim that $P'$ is a supermartingale.
To see this, we consider any step where we clean some vertex $u$, 
for any $v,W$ let $Q'_{v,W}$ denote the updated value of $P'_{v,W}$,
and show $\mb{E}'[Q'_{v,W}] \le P'_{v,W}$.
Write $W' = X_u \cap W$. For $w \in W'$ let $E_w$
denote the event that $uw$ receives colour $c$.
By \cref{lem:clean}, for any $X \sub W'$ we have
$\mb{P}( \bigcap_{w \in X} E_w) \le (L/m)^{|X|}$.
Thus $\mb{E}'[Q'_{v,W}]/P'_{v,W} \le e^{-4|W'|L/m}\sum_{X \sub W'} (eL/m)^{|X|}
= e^{-4|W'|L/m} (1+eL/m)^{|W'|} < 1$, as claimed.
Furthermore, $|Q - P| \le \DD(H)^2 = (dm)^2$,
so the variance proxy is $<ndm(dm)^4 < n(\log n)^{21}$, say.
By Azuma we conclude $\mb{P}(Z'_3 > .2\tT^2 n) < e^{-\sqrt{n}}$, say.

Similarly, we bound $Z'_4$ by a predictor, abusively also denoted 
$P'=\sum_v P'_v$ with notation as for $Z'_3$, but now with each
$P'_{v,W} = \exp[ \sum_{w \in W} ( D(w,c) - 8D(w)/\eps dm - d^{.9} m ) ]$.
We have $Z'_4 \le P'$ as if there is some
independent set $W$ of $d^{.8} m$ neighbours of $v$
with $D(w,c) > d^{.91}m$ for all $w \in W$ then $P'_v=1$.
The starting value of $P'$ is $ne^{- d^{.9} m}$.
To see that $P'$ is a supermartingale, we consider any $P'_v<1$ 
and any term $P'_{v,W}$ updated to $Q'_{v,W}$
by some standard step with $v^*=u$ where $u$ is $c$-full.
This step chooses $u^*=w$ for at most one $w \in W$,
with probability $\le |W| \cdot 2/\eps dm$,
so $\mb{E}'[Q'_{v,W}]/P'_{v,W} \le e^{-8|W|/\eps dm}
 ( (2|W|/\eps dm) \cdot e + (1-2|W|/\eps dm) \cdot 1) <1 $.
The variance proxy is $< n(\log n)^{21}$, so by Azuma
$\mb{P}(Z'_4 > .2\tT^2 n) < e^{-\sqrt{n}}$, say.

Taking a union bound over colours $c \in [m]$,
we deduce $\mb{P}(E_4) < 3mn^{-.9} < n^{-.8}$.
To conclude, we bound the probability of aborting as
$\mb{P}(\tau<\infty) \le \sum_{i=1}^4  \mb{P}(E_i) <\sqrt{\tT}$, as desired.
\end{proof}

\end{document}